\documentclass{amsart}

\usepackage[colorlinks=True, citecolor=blue, linkcolor=blue]{hyperref}
\usepackage{amssymb}
\usepackage{tikz-cd}
\usepackage{algorithm}
\usepackage{algorithmic}
\usepackage{mathtools}
\usepackage{breqn}
\usepackage[nameinlink]{cleveref}

\theoremstyle{definition}
\newtheorem{definition}{Definition}[section]
\newtheorem{example}[definition]{Example}
\newtheorem{remark}[definition]{Remark}
\newtheorem*{assumption}{Assumption}

\theoremstyle{plain}
\newtheorem{theorem}[definition]{Theorem}
\newtheorem{lemma}[definition]{Lemma}
\newtheorem{proposition}[definition]{Proposition}
\newtheorem{corollary}[definition]{Corollary}

\DeclareMathOperator{\rk}{rk}

\DeclareMathOperator{\id}{id}
\DeclareMathOperator{\triv}{Triv}
\DeclareMathOperator{\mwl}{MWL}

\DeclareMathOperator{\NS}{NS}
\DeclareMathOperator{\Aut}{Aut}

\DeclareMathOperator{\PGL}{PGL}

\DeclareMathOperator{\End}{End}

\DeclareMathOperator{\ord}{ord}

\newcommand{\QQ}{\mathbb{Q}}
\newcommand{\FF}{\mathbb{F}}

\newcommand{\CC}{\mathbb{C}}
\newcommand{\ZZ}{\mathbb{Z}}
\newcommand{\PP}{\mathbb{P}}

\newcommand{\NN}{\mathbb{N}}

\renewcommand{\O}{\mathcal{O}}

\newcommand{\R}{\mathcal{R}}
\renewcommand{\S}{\mathcal{S}}

\newcommand{\fp}{\mathfrak{p}}

\newcommand{\lehmer}{l}

\newcommand{\RN}[1]{\textup{\uppercase\expandafter{\romannumeral#1}}}
\newcommand{\even}{\RN{2}}

\makeatletter
\newcommand*{\defeq}{\mathrel{\rlap{%
                     \raisebox{0.3ex}{$\m@th\cdot$}}%
                     \raisebox{-0.3ex}{$\m@th\cdot$}}%
                     =}
\makeatother

\makeatletter
\def\blfootnote{\xdef\@thefnmark{}\@footnotetext}

\title[]{Equations for a K3 Lehmer map}

\author{Simon Brandhorst, Noam D. Elkies}

\address{Simon Brandhorst,
Fakult\"at f\"ur Mathematik und Informatik, Universit\"at des Saarlandes, Campus E2.4, 66123 Saarbr\"ucken, Germany}
\email{brandhorst@math.uni-sb.de}

\address{Noam D. Elkies, Department of Mathematics, Harvard University, Cambridge, MA 02138}
\email{elkies@math.harvard.edu}

\thanks{
Der erste Autor wird Gefördert durch die Deutsche Forschungsgemeinschaft (DFG) – Projektnummer 286237555 – TRR 195.
The first author was funded by the Deutsche Forschungsgemeinschaft (DFG, German Research Foundation) – Project-ID 286237555 – TRR 195.
}

\begin{document}
\begin{abstract}
C.T. McMullen proved the existence of a K3 surface with an automorphism of entropy given by the logarithm of Lehmer's number, which is the minimum possible among automorphisms of complex surfaces.
We reconstruct equations for the surface and its automorphism from the Hodge theoretic model provided by McMullen.
The approach is computer aided and relies on finite non-symplectic automorphisms, $p$-adic lifting, elliptic fibrations and the Kneser neighbor method for $\ZZ$-lattices.
It can be applied to reconstruct any automorphism of an elliptic K3 surface from its action on the Neron-Severi lattice.
\end{abstract}
\maketitle

\blfootnote {{\it 2010 Mathematics Subject Classification}: 37F80,11R06, 14J50, 14J27 14J28, 14Q10} \blfootnote {{\it Key words and phrases: elliptic K3 surface, dynamical degree, Lehmer's number.} }

\addtocontents{toc}{\protect\setcounter{tocdepth}{1}}
\section{Introduction}
The topological entropy $h(g)$ of a biholomorphic map $g\colon S \rightarrow S$ of a compact, connected complex surface $S$ is a measure for the disorder created by repeated iteration of $g$.
It is either zero or the logarithm of a so called Salem number.

That is, it is a real algebraic integer $\lambda > 1$ which is conjugate to $1/\lambda$ and whose other conjugates lie on the unit circle.
The smallest known Salem number is the root $\lambda_{10}= 1.176280818\dots$ of
\[S_{10}(x)=x^{10}+x^{9}-x^{7}-x^{6} - x^5-x^4-x^3+x+1\]
found by Lehmer in 1933.
Conjecturally it is the smallest Salem number, even the smallest algebraic integer with Mahler measure $>1$.

In \cite{mcmullen:rational} C.T. McMullen showed that Lehmer's conjecture holds for the set of entropies coming from automorphisms of surfaces, i. e. $h(g)=\log \lambda(g)$ is either zero or bounded below:
\[h(g) = 0 \quad \mbox{ or} \quad \log \lambda_{10} \leq h(g).\] This can be interpreted as a spectral gap since the dynamical degree $\lambda(g)$ is the largest eigenvalue of $g^*|H^2(S,\CC)$.
If the entropy $h(g)$ is non-zero, then the surface $S$ is birational to a rational surface, a complex torus, a K3 or an Enriques surface \cite{cantat}.
The bottom $\log \lambda_{10}$ of the entropy spectrum can be attained only on a rational surface and on a K3 surface but not on an abelian surface (for trivial reasons) and not on Enriques surfaces \cite{oguiso}.

Explicit equations for a rational surface and its automorphism of minimum entropy are given in \cite{mcmullen:rational}.
The case of K3 surfaces was treated in a series of papers by C.T. McMullen \cite{mcmullen:glue} who first proved the existence of a non-projective K3 surface admitting an automorphism of entropy $\log \lambda_{10}$ and then refined his methods to prove the following theorem.
\begin{theorem}\cite{mcmullen:minimum}
 There exists a complex projective K3 surface $S$ and an automorphism $\lehmer \colon S \rightarrow S$ with minimal topological entropy $h(\lehmer)=\log \lambda_{10}$.
\end{theorem}
The proof proceeds by exhibiting a Hodge theoretic model for $(S,l)$, that is
McMullen constructs a $\ZZ$-lattice $H$ and an isometry $l' \in O(H)$ with spectral radius $\lambda_{10}$ and certain further properties.
Then the strong Torelli-type theorem and the surjectivity of the period map for K3 surfaces guarantee that the Hodge theoretic model $(H,l')$ is induced by a K3 surface $S$ and $\lehmer\in \Aut(S)$ via
an isometry (more precisely a marking) $H^2(S,\ZZ)\cong H$.
However the Torelli-type theorem is non-constructive, so this is a purely abstract existence result.

This work promotes computational methods to reconstruct equations of an automorphism from its Hodge theoretic model. A key ingredient is the constructive treatment of elliptic fibrations on a K3 surface as developed by the second author and A. Kumar \cite{elkies:shimura,kumar,elkies-kumar}. They have the benefit of working in positive characteristic as well.
Our motivating example is to derive explicit equations both for the surface $S$ and its Lehmer automorphism $\lehmer$.
For $n \in \NN$ set $\zeta_n = \exp(2\pi i /n)$.

\begin{theorem}
Let $w$ be a root of
\[w^6 - 2w^5 + 2w^4 - 3w^3 + 2w^2 - 2w + 1 = 0.\]
Let $S_1$ denote the minimal model of the following surface.
\begin{dgroup*}
\begin{dmath*}
 y^2 = x^3 + ax + bt^7 + c \quad \mbox{where}
\end{dmath*}
\begin{dmath*}
a = (-86471w^5 - 19851w^4 - 116626w^3 + 67043w^2 - 125502w + 106947) / 48
\end{dmath*}
\begin{dmath*}
  b = 7  (-w^5 + w^4 + 2w^3 - 3w^2 + 3w - 1)
\end{dmath*}
\begin{dmath*}
  c =(141655682w^5 - 65661512w^4 + 230672148w^3
        - 136877559w^2 + 149096157w - 96818792) / 864
\end{dmath*}
\end{dgroup*}
Equations for a K3 surface $S_6 \cong S_1=S$ and an automorphism $l\colon S_6 \rightarrow S_6$
with entropy $h(l) = \log \lambda_{10}$ are given in \Cref{sect:goodfib} and the ancillary file to \cite{brandhorst-elkies}.
The coefficients lie in $\QQ[\zeta_7,\omega]$ which is a degree $2$ extension of $\mathbb{Q}[\zeta_7]$.
\end{theorem}
\begin{figure}[htbp!]
 \centering{
 \includegraphics[scale=0.14]{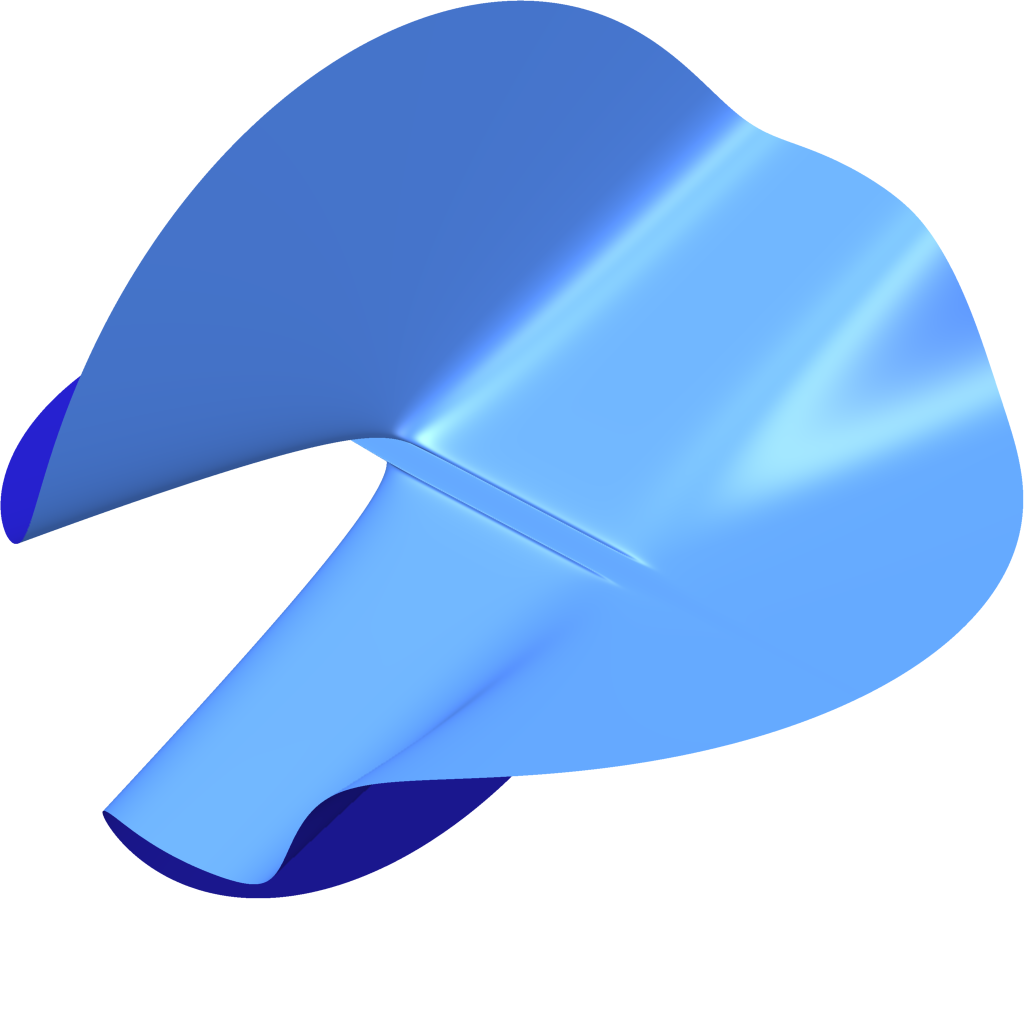}
 \includegraphics[scale=0.14]{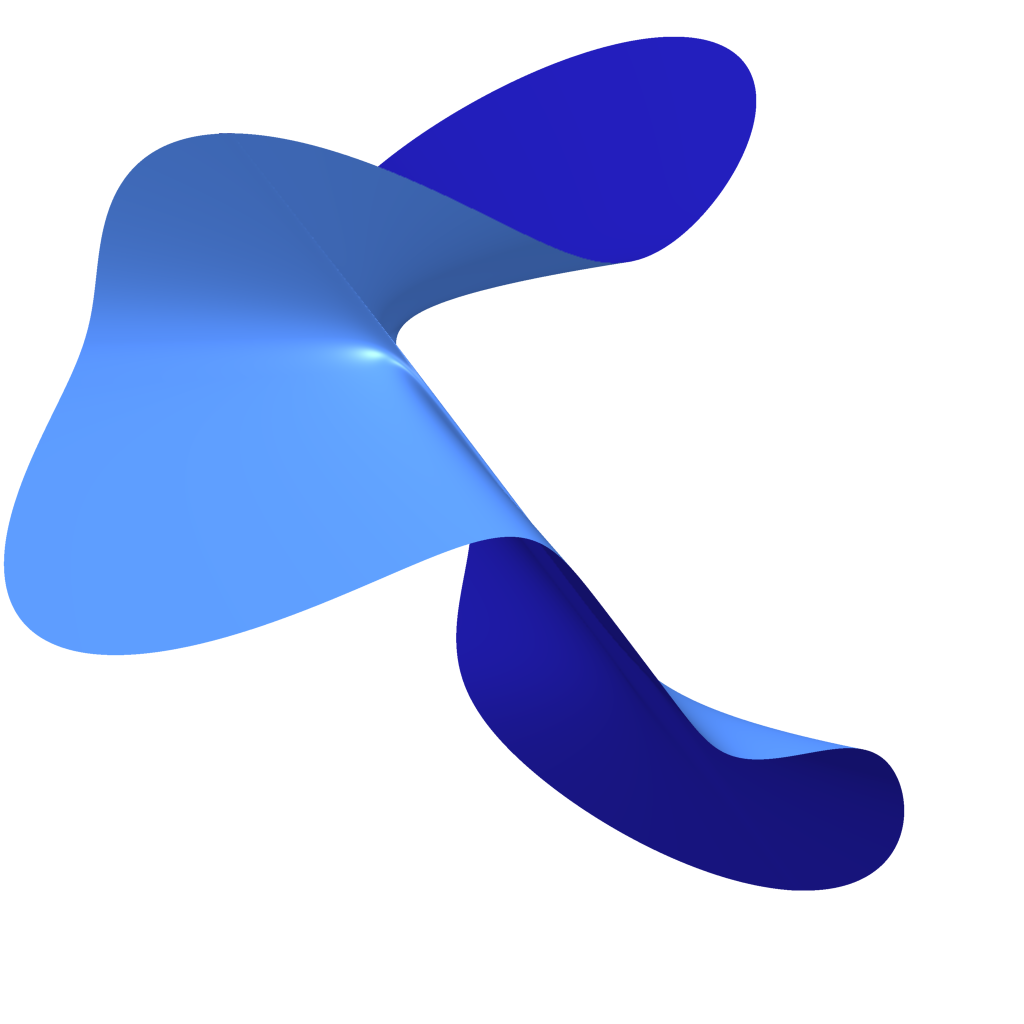}
 }
\caption{The surface $S_1$ admitting an automorphism of minimum entropy.}\label{fig:lehmer}
\end{figure}

\begin{remark}
 \Cref{fig:lehmer} shows the real locus of $S_1$. By \cite[Thm. 2.16]{zhao} Lehmer's map is not defined over the reals. Therefore we cannot plot its orbits. Equations for $S_6$ and $l$ are not printed here since the coefficients in the degree $12$ field are unwieldy. But see \Cref{sect:lehmer} for a (factored) representation of $l$ with coefficients modulo $29$.
\end{remark}

As a corollary we obtain realizations of Lehmer's number in almost all characteristics and on infinitely many supersingular K3 surfaces.
\begin{corollary}
 There exists a K3 surface $S/\overline{\FF}_p$ and an automorphism $l \in \Aut(S/\overline{\FF}_p)$ of dynamical degreee $\lambda_{10}$ for all primes $p \neq 2,3,7$.
 For $p\equiv 1,2,4 \mod 7$, the suface is of height $h=1,3,3$.
 For $p \equiv 3,5,6 \mod 7$ the surface is supersingular.
 If further $p\neq 13$ its Artin invariant is $\sigma = 3, 3, 1$.
\end{corollary}
\begin{proof}
 Our model for the surface $S_6$ over the ring of integers $\O_K$, $K= \QQ[\zeta_7,\omega]$ is of good reduction for all primes not dividing $42$. Further Lehmer's map is defined by explicitly given rational functions $f_i/g_i$ where the coefficients of $f_i,g_i$ lie in $\O_K$. One checks that the coefficient ideal of $g_i$ has prime factors dividing at most $2,3,7$. Thus if $P \leq \O_K$ does not divide $42$, then each $g_i \mod P$ is nonzero, so that $f_i/g_i \mod P$ is a well defined rational function on the reduction of $S_6$ modulo $P$.

 Since the transcendental lattice has rank $6$ and $S$ has an automorphism acting by a primitive $7$-th root of unity \cite[Thm. 2.3]{jang2014} applies and computes the height and Artin invariant.
\end{proof}

In what follows we sketch how we derived the equations.
Since the Lehmer map has positive entropy, it does not preserve any polarization. This means that it is not linear, i.e. we cannot represent $l$ as an element of $\PGL(n,\CC)= \Aut(\PP^n)$ acting on the surface $S \subseteq \PP^n$.

However, from the Hodge theoretic model one infers that
$S$ has complex multiplication by $\QQ[\zeta_{7}]$ which, as we shall see in \Cref{sect:aut}, comes from a non-symplectic automorphism $\sigma$ of order $7$.
This automorphism in fact determines $S$ up to isomorphism. Since it preserves a polarization, the automorphism $\sigma$ is linear, so it is much easier to write down. Indeed, we find a one dimensional maximal family of K3 surfaces with an automorphism of order $7$ in \cite{ast:prime} that must contain $(S,\sigma)$.

We locate the sought for surface inside the family by first reducing it modulo a suitable prime $p$ and then lift the resulting surface to the $p$-adic numbers with a multivariate Newton iteration in \Cref{sect:newton}. The coefficients of $S$ are recovered as algebraic numbers in a degree $6$ subfield of $\QQ[\zeta_{7}]$.

To derive the automorphism from its action on cohomology, we use the theory of elliptically fibered K3 surfaces. In fact the surface $S$ is already represented in terms of a Weierstrass model. K3 surfaces may carry more than one elliptic fibration. We view  elliptic fibrations on our surface $S$ as vertices of a graph and the coordinate changes between them as edges. The resulting graph is in fact closely related to the Kneser neighbor graph of an integer quadratic form. See \Cref{sect:fibrationhop} for the details. Further we give an algorithm how to derive the change of coordinates corresponding to an edge of length $2$.

Now the idea to get the Lehmer map $l$ is the following: let $f\in \NS(S)$ be the class of a fiber. Then $f'=l_*(f)$ is the fiber class of another elliptic fibration on $S$. Both classes are visible on the Hodge theoretic side. Charting a path in the neighbor graph from $f$ to its image $f'$ yields a birational map $W \dashrightarrow W'$ between two Weierstrass models of $S$. In fact they must (up to Weierstrass isomorphisms) be the same since we know that $l$ is an isomorphism. Thus this isomorphism (which on an affine chart is nothing but a change of coordinates) can be seen as an automorphism $\tilde{l} \colon S \rightarrow S$. By construction $(\tilde{l} \circ l^{-1})_*(f)=f$ preserves the class of a fiber. Such automorphisms act on the base $\PP^1$ or are fiberwise translations. In any case they are easily controlled which leads us to a Lehmer map $l$ in \Cref{sect:lehmer}.

The calculations were carried out using the computer algebra systems Pari-GP \cite{pari}, SageMath \cite{sage}, Singular \cite{singular}.

\subsection*{Acknowledgements}
The authors would like to thank the Banff International Research Station for Mathematical Innovation and Discovery (BIRS) for hosting the conference 'New Trends in Arithmetic and Geometry of Algebraic Surfaces' in 2017 where the idea for this work was conceived. We thank Curtis T. McMullen and Matthias Sch\"utt for comments and discussions.

\section{Kneser's neighbor method and fibration hopping.}\label{sect:fibrationhop}
In this section we review a connection between elliptic fibrations on a given K3 surface
and the neighboring graph of the genus of a quadratic form found by the second author.
It is described in \cite[Appendix]{kumar}, \cite[Sect. 5]{elkies-kumar} and used in \cite{elkies-schuett}, \cite{kumar:hilbert-genus2}.
In this work we take an algorithmic point of view.

\subsection{The lattice story}
A lattice ( or $\ZZ$-lattice) consists of a finitely generated free abelian group $L$ together with a non-degenerate symmetric bilinear form
\[\langle \cdot \, , \cdot \rangle \colon \;\; L \times L \longrightarrow \ZZ.\]
If the bilinear form is understood, we omit it from notation and simply call $L$ a lattice.
Further we abbreviate $\langle x,x \rangle=x^2$ and $\langle x,y \rangle = x.y$ for $x,y \in L$. The lattice $L$ is called even if $x^2 \in 2\ZZ$ for all $x \in L$. Otherwise it is called odd. Vectors $x \in L$ with $x^2 = -2$ are called roots.
For $L=\ZZ^n$ we denote the lattice by its gram matrix. For a subset $A \subset L$
we denote by $A^\perp=\{x \in L : x.A=0\}$ the maximal submodule orthogonal to $A$.
An isometry of lattices is an isomorphism of abelian groups preserving the bilinear forms.
We say that two lattices belong to the same \emph{genus}, if their completions $L\otimes \QQ_{\nu}$ at all places $\nu$ of $\QQ$ are isometric. The symbol $\ZZ_p$ denotes the $p$-adic numbers.
See \cite{kneser} for an introduction to quadratic forms and their neighbors.

In this section we review a constructive version of the following theorem. See \cite[4.1]{durfee} for a related proof.
\begin{theorem}\label{thm:U+L}
Let $U\cong \left[\begin{smallmatrix} 0&1\\ 1&0 \end{smallmatrix}\right]$ be a hyperbolic plane.
Two lattices $L$, $L'$ are in the same genus if and only if $U \oplus L \cong U \oplus L'$.
\end{theorem}

\begin{definition}
Let $L,L'$ be lattices in the same quadratic space $V$ and $p$ a prime number.
We say that $L$ and $L'$ have $p$-distance $n$ if
$p^n=[L:L \cap L']=[L':L \cap L']$.
Then we call them $p^n$-neighbors.
\end{definition}

\begin{assumption}
For the rest of this section let $p\nmid \det L$ be a prime and $\rk L\geq 3$.
\end{assumption}
Note that $p$-neighbors have the same determinant. Indeed, since $p \nmid \det L$,
both $L\otimes \ZZ_p$ and $L' \otimes \ZZ_p$ are unimodular, of the same rank and determinant. Thus they are isometric if $p\neq 2$ and for $p=2$ they are isometric if and only if both are odd or both are even. For primes $q
\neq p$ we have $L\otimes \ZZ_q = L' \otimes \ZZ_q$.
Thus any two even (resp. odd) $p$-neighbors lie in the same genus.
The following theorem works in the converse direction.
At a first read the reader may ignore the difference between genus and the so called spinor genus since they usually agree.
\begin{theorem}
Any two classes in the spinor genus of $L$ are connected by a sequence of $p$-neighbors. If $p=2$ and $L$ is even, then this sequence can be chosen to consist of even lattices only.
\end{theorem}
\begin{proof}
 This is a consequence of \cite[28.4]{kneser} which in fact works with weaker assumptions.
\end{proof}
For indefinite genera (of rank at least $3$) the spinor genus consists of a single isometry class and the genus consists
of $2^s$ ($s \in \NN$, usually $s=0$) spinor genera.
In the definite case, the number of isometry classes in a genus is still finite but in general one has to use algorithmic methods to enumerate them. The standard approach uses the previous theorem: it explores the neighboring graph by passing iteratively to neighors.
This rests on the following explicit description of $p$-neighbors.
\begin{lemma}\cite[28.5]{kneser}\label{lem:neighbor}
 Let $v \in L \setminus pL$ with $v^2 \in p^2\ZZ$, then
 \[L(v)= L_v + \ZZ\tfrac{1}{p}v \mbox{ with } L_v = \{x \in L \mid x.v \in p \ZZ\}\]
 is a $p$-neighbor and every $p$-neighbor is of this form.
\end{lemma}
Neighbors turn out to be useful in the hyperbolic case as well.
Let $N$ be a lattice. We call an element $x \in N$ primitive, if for $y \in N$ and $n \in \ZZ$, $x = ny$ implies that $n = \pm 1$.

Let $f_1 \in N$ be primitive with $f_1^2=0$.
In the case that $f_1$ can be completed to a hyperbolic plane by $e_1 \in N$ with $e_1^2=0$ and $e_1.f_1=1$, then we have
$f_1^\perp / f_1 \cong \{f_1,e_1\}^\perp$.
Suppose that $f_1.f_2=p$. Since $p$ $\nmid \det N$, we find $e' \in N$ with $p \nmid f_1.e'$.
This implies that we can complete $f_2$ to a hyperbolic plane.

\begin{lemma}\label{lem:pneighbors}
 Let $N$ be a lattice and $f_1,f_2 \in N$ be primitive with $f_1^2=f_2^2=0$ and $f_1.f_2=p$. Denote by $L_i$ the image of the orthogonal projection
 \[f_i^\perp/\ZZ f_i \longrightarrow \{f_1,f_2\}^\perp \otimes \QQ.\]
 Then $L_1$ and $L_2$ are $p$-neighbors in the quadratic space
 $\{f_1,f_2\}^\perp \otimes \QQ$.
\end{lemma}
\begin{proof}
The lattices $L_1$ and $L_2$ are $p$-neighbors since their intersection
is precisely $\{f_1,f_2\}^\perp \subseteq N$ which is of index $p$ in each:
we can compare the determinants of $L_i \cong f_i^\perp/\ZZ f_i\cong \{f_i,e_i\}^\perp$ and $\{f_1,f_2\}^\perp$; the first one is $\det N$, and since $p$ does not divide $\det N$, the second one is $-p^2 \det N$ (use that the completions $N\otimes \ZZ_p$ and $\ZZ_q f_1 \oplus \ZZ_q f_2$, $q\neq p$ are unimodular).
\end{proof}
In view of \Cref{lem:neighbor} we make this explicit.
\begin{lemma}\label{lem:neighbor2}
 Let $L_1$ be an even lattice and $v \in L_1 \setminus pL_1$ with $v^2 \in 2p^2\ZZ$.
 Fix a hyperbolic plane $U_1 = \ZZ e_1 \oplus \ZZ f_1$ with $(e_1,f_1)=1$, $e_1^2 = f_1^2=0$ and consider $N=U_1 \oplus L_1$.
 Set
 \[f_2 = -\frac{v^2}{2p} f_1 + pe_1 + v\]
 and choose some $e_2 \in U_1 \oplus L_1$ with $e_2^2=0$ and $(e_2,f_2)=1$.
 Then the orthogonal complement $L_2$ of $U_2=\ZZ e_2 \oplus \ZZ f_2$ is isomorphic to the $p$-neighbor $L_1(v)$ of $L_1$.
\end{lemma}
\begin{proof}
By construction $f_2$ is primitive, $f_2^2=0$ and $f_2.f_1=p$, so \Cref{lem:pneighbors} applies to give that the orthogonal complements are neighbors.
Restricting the orthogonal projection to $L_1$ and $L_2$ gives the isometries
\[\pi_1\colon L_1 \rightarrow \{f_1,f_2\}^\perp\otimes \QQ,\qquad x\mapsto x - \frac{x.f_2}{p} f_1,\]
and
\[\pi_2 \colon L_2 \rightarrow \{f_1,f_2\}^\perp \otimes \QQ, \qquad y \mapsto y - \frac{y.f_1}{p} f_2.\]
Indeed, since $f_1$ is isotropic and orthogonal to $L_1$, $\pi_1$ preserves the bilinear form and likewise does $\pi_2$.
To see that $\pi_1(L_1(v)) = \pi_2(L_2)$,
one can calculate that $\pi_1(v)$ is the image of $\pi_1(v) - (e_2.\pi_1(v)) f_2 \in L_2$ under $\pi_2$.
\end{proof}

\subsection{The geometric story}
See \cite{schuett2010} for a survey on elliptic fibrations.
Let $S$ be a K3 surface over an algebraically closed field $K=\bar K$.
For simplicity we exclude the possibility of quasi-elliptic fibrations, by assuming that the characteristic of $K$ is not $2$ or $3$.
\begin{definition} A \emph{genus one fibration} on $S$ consists of a morphism $\pi\colon S \rightarrow \PP^1$ whose generic fiber is a smooth curve of genus one over the base.
An \emph{elliptic fibration} is a genus one fibration equipped with a distinguished section
$O\colon \PP^1 \rightarrow S$ with $\pi \circ O = \id_{\PP^1}$.
\end{definition}
Rational points of the generic fiber correspond to sections of the fibration and vice versa.
The zero section $O$ defines a rational point on the genus one curve $S/K(\PP^1)$ over the function field $K(\PP^1)$ of $\PP^1$.
This turns the generic fiber into an elliptic curve with zero given by $O$.
Such a curve has a Weierstrass model
\[y^2 + a_1 xy + a_3 y = x^3 + a_2 x^2 + a_4 x + a_6\]
with $a_i(t,s) \in K[\PP^1]$ homogeneous of degree $2i$. This defines a normal surface in weighted projective space $\PP(1,1,4,6)$ whose minimal model is the K3 surface $S$.
Here $s,t$ have weight $1$, $x$ has weight $4$ and $y$ weight $6$.

Denote by $f$ and $o \in \NS(S)$ the algebraic equivalence classes of a fiber $F$ and the zero section $O$.
Their intersection numbers are $e^2=-2$, $f^2=0$ and $e.f=1$. Thus they span a hyperbolic plane
$\left[\begin{smallmatrix} 0&1\\ 1&0 \end{smallmatrix}\right]\cong \left[\begin{smallmatrix} 0&1\\ 1&-2 \end{smallmatrix}\right]\cong U_\pi \subseteq \NS(S)$.
\begin{definition}
We call $U^\perp_\pi$ the \emph{frame lattice} and
$U_\pi \oplus \R(U^\perp_\pi)$ the \emph{trivial lattice} where $\R(U^\perp_\pi)$ is the sublattice spanned by the roots of $U_\pi^\perp$.
\end{definition}
The trivial lattice gives information on the reducible singular fibers. For instance the reducible fibers not meeting the zero section
form a fundamental root system for the root lattice $\R(U^\perp_\pi)$.

The Mordell-Weil group is the group of sections of the fibration. It comes equipped with the height pairing $h$ which is a $\QQ$-valued bilinear form and turns it into the \emph{Mordell-Weil lattice} $\mwl(\pi)$.
The Mordell-Weil lattice is isomorphic to the image of the orthogonal projection
$\NS(S) \rightarrow (\triv(\pi)^\perp)^\vee$ equipped with the negative of the intersection form (cf. \cite[Lemma 8.1]{shioda}). In fact the Mordell-Weil group is isomorphic to $\NS(S)/\triv(\pi)$ (cf. \cite[Theorem 1.3]{shioda}).
Addition in the lattice indeed corresponds to the group law on the elliptic curve.

A K3 surface may admit several elliptic fibrations. They can be detected in the N\'eron-Severi lattice:
\begin{theorem}\cite[Paragraph 3]{shafarevic:torelli}
Let $S$ be a K3 surface and $f \in \NS(S)$ a primitive nef divisor class with $f^2=0$. Then the complete linear system $|f|$ induces
a genus one fibration $\pi_{|f|}\colon S \rightarrow \PP^1$.
\end{theorem}
We call such a class an \emph{elliptic divisor class}.
If there is $e \in \NS(S)$ with $e^2=-2$ and $e.f=1$, then $\pi_{|f|}$ is indeed an elliptic fibration.
\[\{\mbox{genus one fibrations } \pi \colon S \to \PP^1\} \xleftrightarrow{1:1} \{\mbox{elliptic divisor classes } f \in \NS(S)\}\]
\begin{remark}\label{rmk:weyl}
 If $f \in \NS(S)$ is merely primitive with $f^2 = 0$, then there is an element $w \in W(\NS(S))$ of the Weyl group with $\pm w(f)$ nef.
 Hence \Cref{thm:U+L} shows that every lattice in the genus of some frame lattice $U_\pi^\perp$ is a frame lattice
 $U_{\pi'}^\perp$ for some elliptic fibration $\pi'\colon S \rightarrow \PP^1$.
\end{remark}
\begin{example}
 The class $o+f$ has square $0$, but it is not nef since $(o+f).o=-2+1<0$.
\end{example}
\begin{definition}
Let $\pi_1,\pi_2$ be two elliptic fibrations with fiber classes $f_1$, $f_2$.
We call the fibrations $n$-neighbors if their frame lattices project to $n$-neighbors in $\{f_1,f_2\}^\perp$.
\end{definition}
Suppose that $n$ is coprime to $\det \NS(S)$. By the previous section, two elliptic fibrations are $n$-neighbors if $f_1.f_2$ equals $n$.
\begin{remark}
The traditional way to produce elliptic divisors, is to find a configuration of $(-2)$ curves on $S$ whose intersection graph is an extended Dynkin-diagram of type $\tilde{A}$, $\tilde{D}$, or $\tilde{E}$. The corresponding isotropic class is automatically nef.
The isotropic class $f_2$ constructed in \Cref{lem:neighbor2} may not be nef.
Often this can be compensated with an element $w$ of the Weyl group as in \Cref{rmk:weyl}. However in general $f_1.w(f_2)$ may be different from $f_1.f_2$.
\end{remark}

\subsection{Computing $2$-neighboring fibrations}
 In this subsection we give an algorithm to compute the linear system $|f'|$ of an elliptic divisor starting from a Weierstrass model of a $2$-neighbor $f$.

Let $\pi$ be an elliptic fibration on $S$, $F$ a fiber and $D$ a divisor on $S$. Then $D$ is called vertical if $F.D =0$. If $D$ is effective, this means that it is contained in some fiber of $\pi$.
\begin{lemma}\cite[Lemma 5.1]{shioda}\label{lem:gensNS}
 Every divisor $D$ of an elliptic K3 surface is linearly equivalent to a divisor of the form
 \[D \sim (d-1)O + P + V\]
 for $d= D.F$, $P$ some section (possibly the zero section) and $V$ a vertical divisor.
\end{lemma}

Suppose $F'$ is an elliptic divisor with $F.F'=2$.
Then
\[F' \sim O + P + V\]
with $V$ vertical but not necessarily effective.
We want to compute this linear system.

Since $V$ is vertical, there is $k\geq 0$ such that the class of $kF - V$ is effective.
To determine $k$ one uses extended Dynkin diagrams and their isotropic vectors to represent $f$ as a linear combination of fiber components of the respective reducible fibers.
Our strategy is to first compute the larger linear system $|O + P + k F|$ and then to figure out the equations of the linear subspace $|O + P + V|$ of $|kF + O + P|$.\\

Let $(X,Y)$ be affine coordinates of the ambient affine space of a Weierstrass model of $E/K(t)$. We set $\deg X = 4$ and $\deg Y=6$.
Suppose that $P=(P_X,P_Y) \in E$ is \emph{not a torsion section}. (See \cite{kumar} for this case.)
Then $1$ and $(Y+P_Y)/(X-P_X)$ are a basis of global sections of $\O_E(F')$.
Thus every section of $\O_S(F')$ is of the form
$a(t) + b(t)(Y+P_Y)/(X-P_X)$ for some $a(t),b(t) \in K(t)$.
The divisor $V$ gives conditions on the zeros and poles of $a$ and $b$.
\begin{remark}
 The dimension of the linear system is predicted by the Riemann-Roch formula and Serre duality.
 Indeed, for an effective divisor $D$ on a K3 surface $S$ we have
 \[h^0(S,\O_S(D))=h^0(S,\O_S(D))+h^0(S,\O_S(-D)) = 2 + \tfrac{1}{2} D^2+h^1(S,\O_S(D)).\]
 The exact sequence
 $0 \rightarrow \O_S(-D) \rightarrow \O_S \rightarrow \O_D \rightarrow 0$
 shows that
 $h^1(S,\O_S(D))=h^1(S,\O_S(-D)) = \dim h^0(D,\O_D)-1$ which is zero for $D= O + P + lF$ since $D$ is numerically connected (cf. \cite[Lem. 2.2 and 3.4]{donat}).
\end{remark}
We regard $K(t) \subseteq K(S)$ and write the elements $e \in K(t)$ in the form $e=e_n/e_d$ where
the numerator $e_n$ and denominator $e_d$ are co-prime elements of the polynomial ring $K[t]$.
For $P$ we write $x_n=(P_X)_n$ and $x_d=(P_X)_d$ and likewise  $y_n =(P_Y)_n$, $y_d=(P_Y)_d$.
\begin{lemma} Let $O\neq P \in E$ be a section. Then
 \[2O.P = \deg (X x_d -x_n) - 4 = \tfrac{2}{3}(\deg (Y y_d -y_n) - 6)\]
\end{lemma}
\begin{proof}
For a start note that $x_d^3=y_d^2$.
The intersection multiplicity in the chart $s \neq 0$,
is
\[\deg (y_d/x_d) =  \tfrac{1}{2}\deg x_d.\]
 We change to the chart given by $X = \tilde{X}/s^4$, $Y = \tilde{Y}/s^6$, $t=1/s$.
 In these coordinates the section $P$ is given by \[(\tilde{x},\tilde{y})=(s^4 x(\tfrac{1}{s}),s^6y(\tfrac{1}{s})).\]
 Hence the valuation at $s$ is given by
 \[\nu_s(\tilde{x}) = 4 + \deg x_d - \deg x_n\]
 Similarly we have $\tilde{x}_d^3 = \tilde{y}_d^2$
 and so the intersection multiplicity at $s=0$ is
 \[\max\{0,-\nu_s(\tilde{y}/\tilde{x})\}=\tfrac{1}{2}\max\{0, \deg x_n - \deg x_d -4\}.\]
 Combined this results in
 \begin{eqnarray*}
2 O.P &=& \deg x_d + \max\{0,\deg x_n - \deg x_d - 4\} \\
&=& \max\{\deg x_d+4,\deg x_n\}-4 = \deg(Xx_d-x_n)-4. \end{eqnarray*}
\end{proof}

\begin{proposition}\label{prop:linear-system}
Let $F$ be the divisor on $S$ given by $\{t=0\}$, $P$ a non-torsion section and $k\geq 0$.
Then the elements of  $H^0(S,O + P + kF)$ are given by
\[\frac{a(t)(X x_d-x_n) + b(t)(Y y_d+y_n)}{t^kx_d(X x_d-x_n)}\]
with $a,b \in K[t]$, $ 2O.P= \max(\deg x_d , \deg x_n-4)$ and
\begin{enumerate}
 \item $\deg a \leq k + 2 O.P$,
 \item $\deg b \leq k + 2 O.P -2 - \tfrac{1}{2}\deg x_d$,
 \item $\deg (a x_n - b y_n) \leq k +2 O.P + 4 +\deg x_d$,
 \item $a x_n - b y_n \equiv 0 \mod x_d$.
 \end{enumerate}
 This results in a solution space of dimension $h^0(S,O + P + kF) = 2k +O.P$.
 \end{proposition}
\begin{proof}
The condition that
\[u=\tilde{a}(t) + \tilde{b}(t)(Y+P_y)/(X-P_x) \in K(X,Y,t)\]
represents an element of $H^0(S,O + P +k F)$
means that the denominator of $u$ is divisible at most by $t^k(X x_d-x_n)$. Writing $u$ as a reduced fraction yields the following form
%
%
with $a,b \in K[t]$
\begin{eqnarray*}
 u 
   &=& \frac{a x_d X + b y_dY - (a x_n - b y_n)}{t^kx_d(X x_d-x_n)}.
 \end{eqnarray*}
Conditions (1--3) assure that there is no pole at $t=\infty$:
with $\deg X = 4$, $\deg Y=6$ this means that $\deg u \leq 0$.
Condition (4) assures that we can reduce $x_d$ from the fraction.

We compute the dimension of the linear system.
%
Conditions (1) and (2) result in a space of dimension $2k+4 O.P -\tfrac{1}{2}\deg x_d$.
Let $\tilde{a}$, $\tilde{b}$ be general elements of this space.
The rank of condition (3) is
\begin{eqnarray*}
& &\max\{0,\deg (\tilde{a}x_n - \tilde{b}y_n) - (k + 2O.P + 4 +\deg x_d)\}\\
&=&\max\{0,\deg x_n - \deg x_d - 4, \deg y_n - \deg y_d - 6)\}\\
&=&  \tfrac{3}{2}(2 O.P - \deg x_d)
\end{eqnarray*}
Condition (4) is independent and has rank $\deg x_d$.

This gives the total number of solutions
\begin{eqnarray*}
&   & (2k+4 O.P - \tfrac{1}{2}\deg x_d) -\tfrac{3}{2}(2O.P - \deg x_d)-\deg x_d\\
& = & 2k + O.P\\
& = & 2+ \tfrac{1}{2}(kF + O + P)^2
\end{eqnarray*}
as predicted by the Riemann-Roch formula.
\end{proof}

 Next we have to cut down the linear system $|kF + O + P|$ to the smaller system  $|O + P+V|$. To this end we recall some concepts from commutative algebra.
For ideals $I,J$ of a Noetherian ring $R$ let
$I:J = \{x \in R \mid xJ \subseteq I \}$ be the ideal quotient.
Recall that an ideal $I$ is called primary if $xy \in I$ implies that $x \in I$ or $y^n \in I$ for some $n$.
Let $I^{(n)}$ be the $n$-th symbolic power of $I$.
For a prime ideal $P$, the symbolic power $P^{(n)}$ is the smallest $P$-primary ideal containing $P^n$.
The localization of $R$ at $P$ is denoted by $R_P$.
\begin{lemma}\label{lem:order}
 Let $f\in R$ be an element of a Noetherian ring $R$,
 $0\leq w \in \ZZ$, $P \leq R$ a prime ideal with $R_P$ a discrete valuaton ring.
 Then the following are equivalent:
 \begin{enumerate}
  \item  $\ord_P(f) \geq w$,
  \item $f \equiv 0 \mod P^{w}R_{P}$,
  \item $f \equiv 0 \mod P^{(w)}$,
  \item $P^w:fR \not \subseteq P$.
 \end{enumerate}
\end{lemma}
\begin{proof}
 (1) $\iff$ (2):  $R_P$ is a discrete valuation ring with valuation $\ord_P$.

 (2) $\iff$ (3):
 Let $f \in R$ with $f \in P^{w}R_{P}$.
 This means that there is $s \in R\setminus P$ with $sf \in P^{w}$.
 Since $P^w \subseteq P^{(w)}$ and the latter is $P$-primary, this implies that $f \in P^{(w)}$. Hence $R \cap P^{w}R_{P} \subseteq P^{(w)}$.
 It remains to show that $R \cap P^{w}R_{P}$ is $P$-primary. Then, by minimality,
$ R \cap P^{w}R_{P} \supseteq P^{(w)}$.
 Let $f,g \in R$ with $fg \in R \cap P^{w}R_{P}$, i.e. there exists an $s \in R\setminus P$ with $sfg \in P^{w}$. Since $P$ is prime and $s \notin P$, $g \in P$ or $f \in P$. Thus $f^w \in P^w$ or $g^w \in P^w$.

 (3) $\implies$ (4): Suppose that $f \in P^{(w)}=R \cap P^{w}R_{P}$. Then there is $r \in R\setminus P$ with $rf \in P^{w}$ which implies that $rfR \subseteq P^{w}$. Hence $r$ is an element of  $P^w:fR$ not contained in $P$.

 (4) $\implies$ (2): Take $r \in (P^w:fR) \setminus P$, so $rfR \subseteq P^w$. Since $r$ is a unit in $R_P$, this gives $rfR_P=fR_P \subseteq P^wR_P$.
\end{proof}

The following remark describes how to obtain linear equations for the subspace $H^0(S,O + P + V)$ in $H^0(S,kF + O + P)$.
\begin{remark}\label{rmk:linear-system}
 Let $D$ be a vertical prime Weil divisor on $S$.
 In practice we obtain $D$ as exceptional divisor coming from a blowup
 during a minimal resolution of the Weierstrass model (e.g. by Tate's algorithm \cite{tate}).
 Choose some chart $U\subset S$ intersecting $D$. And let $W \subset S$ be a Weiertrass chart.
 We represent $D$ as the pair $(\psi_U,D_U)$ where $\psi_U \colon K(W) \rightarrow K(U)$ is the (rational) change of coordinates and $P_U\leq K[U]$ is the prime ideal giving the Weil divisor $D|_U$.

 Let $\phi_1,\dots \phi_n$ be a basis of $H^0(S,kF + O + P)$. The linear equations cutting out the subspace $H^0(S,O + P + V)$ are of the form $\ord_D(\phi)\geq v\in \ZZ$. To find them let $\phi = \sum_{i=1}^n \alpha_i \phi_i$, $(\alpha_i \in K)$ and write $\psi_U(\phi)=f/g$ for some $f \in K[U]$ depending linearly on the $\alpha_i$ and $g \in K[U]$ some fixed common denominator of the $\phi_i$. In particular, $g$ does not depend on the $\alpha_i$.
 By \Cref{lem:order} the condition $\ord_D(f) \geq v + \ord_D(g)=:w$ is equivalent to $f \equiv 0\mod P_U^{(w-v)}$.
 This defines a linear map
 \[H^0(S,kF + O + P) \longrightarrow k[U]/P_U^{(w-v)}, \quad \phi=f/g \mapsto f + P_U^{(w-v)} \]
 whose kernel is $H^0(S,kF + O + P+V)$.
\end{remark}

The computation of $H^0(S,\O_S(f'))$ allows us to explicitly give a morphism $S \to \PP^1$ whose generic fiber is a curve of genus one.
In order to continue fibration hopping from the newly found fibration, one (searches and) chooses
a point of small height and transforms the genus one curve to minimal Weierstrass form (see e.g. \cite{kumar} for formulas; most computer algebra systems have functionality for this).

To proceed we need generators of the Mordell-Weil group.
Generators for $\NS(S)$ can be obtained via push forward from the previous model.
However, they will in general not be sections but multi-sections, i.e. an irreducible curve $C \subseteq S$ with $C.F > 1$. From a multi-section one obtains a section by taking the fiberwise trace.

\begin{lemma}
 Let $\mathfrak{p}\leq K(t)[x,y]$ be the defining ideal of an irreducible multi-section in a Weierstrass chart given by $y^2-f(x)$. Then $\mathfrak{p} = (y-h(x),g(x))$ or $\mathfrak{p}= (y^2-f(x),g(x))$ for $h(x),g(x) \in K(t)[x]$.
\end{lemma}
\begin{proof}
Let $\fp \cap K(t)[x] = (g(x))$. Then either
$\fp = (y^2-f(x),g(x))$ or we find $r(x)y - \tilde{h}(x) \in \mathfrak{p}$ with $r(x)$ and $g(x)$ coprime. But then we find $s,s' \in K(t)[x]$ with $sr+s'g=1$ so that
$y-s\tilde{h}(x)=:y-h(x) \in \mathfrak{p}$.
Since $(y-h(x),g(x))$ is a prime ideal contained in $\mathfrak{p}$ and containing $y^2-f(x)$, it must be equal to $\mathfrak{p}$.
\end{proof}
If the ideal of the multi-section $D$ is $\mathfrak{p}= (y^2-f(x), g(x))$, then the fiberwise trace is the zero section since $D$ is linearly equivalent to $2(\deg g(x))O$. Otherwise it is computed by the following simple \Cref{alg:trace}.
\begin{algorithm}[h]
\caption{Fiberwise trace}\label{alg:trace}
\begin{algorithmic}
\REQUIRE a Weierstrass model $y^2-f(x)$,\\ a multi-section $(g(x),y-h(x))\subset K(t)[x,y]$
\ENSURE $(x-x(t),y\pm y(t))$ a section which is up to sign the fiberwise trace
\STATE $h:= h \mod g$
\WHILE{$\deg g > 1$}
    \STATE $g:= (h^2 - f)/g$
    \STATE $h:= h \mod g$
\ENDWHILE
\RETURN $(g,h)$
\end{algorithmic}
\end{algorithm}
\begin{proof}[Proof of \Cref{alg:trace}]
By assumption $y^2 - f \in (g,y-h)$.
Thus $h^2-f \in (g)$ which means that $(h^2-f)/g \in K(t)[x]$.
Since the degree of $(h^2-f)/g$ is bounded by $2 \deg h - \deg g < \deg g$, the procedure terminates.
Let $D_1, D_2$ be the divisors defined by $(g,y-h)$ and $((h^2-f)/g, y-h)$.
The divisor $D_1+D_2 - 2 \deg(h) O$ is linearly equivalent to zero.
It is the divisor of the rational function $(y/z-h(x/z))$.
Hence the output section $P$ satisfies a linear equivalence of the form
$D_1 \pm P + n O \sim 0$, i.e. it is up to sign the trace.
\end{proof}

\subsection{A strategy for fibration hopping.}\label{sect:strategy}
We summarize the method of fibration hoppping highlighting practical aspects. It is applied in \Cref{sect:lehmer} to obtain the Lehmer map.
Start with the following data:
\begin{itemize}
 \item
a minimal Weierstrass equation $W: y^2 - x^3-a_2 x^2 - a_4 x - a_6$ defining an elliptic fibration $\pi\colon S \rightarrow \PP^1$ of a K3 surface with fiber class $f \in \NS(S)$;
\item a $\ZZ$-basis $B$ of $\NS(S)$ consisting of sections and fiber components;
\item an elliptic divisor class $f' \in \NS(S)$ with $f.f'=2$ written as a linear combination of the basis.
\end{itemize}
Assume that $f'$ admits a section, i.e. $f'.\NS(S)=\ZZ$, otherwise stop at step (9).
\begin{enumerate}
  \item Compute the intersection matrix of the basis and the height pairing of the sections.
  \item Find a representative $F'=O + P + V \in f'$ as in \Cref{lem:gensNS}.
  \item Compute $P$ using addition in the Mordell-Weil group.
  \item Find $k \in \NN$ with $O + P + V \leq O+P +kF=D$.
  \item Compute the linear system $H^0(S,D)$ as in \Cref{prop:linear-system}.
  \item Resolve the singularities of the Weierstrass model using Tate's algorithm and represent fiber components by a pair $(\phi_U: k(W) \rightarrow K(U), P_U)$ with $P_U \leq K[U]$ the defining ideal.
  \item Cut out the linear subspace $H^0(S,O + P + V)$ of $H^0(S,D)$ using \Cref{rmk:linear-system}.
  \item Choose two elements $\phi_0,\phi_1 \in H^0(S,O + P + V)$ and set $u = \phi_0/\phi_1$.
  \item Solve for $y=y(u,x,t)$ (assuming that $P$ is not $2$-torsion, see \cite[39.1]{kumar} for the general case), substitute into the Weierstrass equation, cancel a common factor and absorb square factors into $x$, to obtain an equation of the form $x^2 = g(u,t)$ of degree $3$ or $4$ in $t$.
  \item Search a $K(u)$-rational point of small height. This can be done by pushing forward suitable divisors from $S$, or by exhaustive search modulo a prime and $p$-adic lifting.
  \item Use the point to obtain a Weierstrass model (see e.g. \cite{kumar,elkies-kumar}).
  \item Use Tate's algorithm to obtain a globally minimal Weierstrass model $W'$ and its singular fibers.
  \item Let $\psi_W: W \dashrightarrow W'$ be the birational change of coordinates between the Weierstrass charts. Push forward fiber components and sections, to obtain multi-sections of the new fibration $\pi'$ and turn them into sections by taking the fiberwise trace with \Cref{alg:trace}.
  \item Compute the height pairing and LLL-reduce the gram matrix to obtain a basis of short vectors and the corresponding sections of small height.
  \item Choose a basis $B'$ of $\NS(S')$ consisting of fiber components and sections.
  \item Pushforward fiber components and sections of $S$ not contained in the indeterminacy locus of $\psi_W$ and
  compute the basis representation w.r.t to $B'$ using the intersection pairing; stop when sufficient information to recover the matrix representation of the pushforward $\psi_*\colon \NS(S) \rightarrow \NS(S')$ in the bases $B$ and $B'$ is obtained.
  For efficiency this is best done over a finite field with subsequent $p$-adic lifting.
\end{enumerate}
\begin{remark}
 The following sanity checks may help the reader to avoid common errors when reproducing the strategy:
\begin{itemize}
\item Make sure the labeling of the exceptional divisors in step (6) matches that of the basis $B$.
\item If the linear subspace $|O + P + V|$ in step (7) is zero, double check that the divisor $f'$ is actually nef.
\item Use the formulas for the Jacobian of a genus $1$ curve to obtain $W'$ (this does not yield the transformation) and compare with your result.
\item Compute the singular fibers of $W'$, compare with the root sublattice of $f^{'\perp}$.
\item Push forward some extra sections of the fibration in step (16) using equations. Then compare with the matrix representation of $\psi_*$.
\end{itemize}
\end{remark}

\section{Finding the surface}\label{sect:aut}
In this section we derive equations for the K3 surface $S$ carrying an automorphism of minimal entropy. We begin by fixing our notation for K3 surfaces.

Let $X$ be a complex K3 surface and $g$ an automorphism. Write $\CC\omega= H^0(X,\Omega_X^2)$ with $\omega$ a non-zero $2$-form. Then $g^*\omega = \delta(g)\omega$ for some $\delta(g) \in \CC$. If $X$ is projective, then $\delta(g)$ must be a root of unity.
We call $g$ symplectic if $\delta(g)=1$ and non-symplectic otherwise.
Recall that we denote by $\NS(X)$ the Neron-Severi lattice and by $T(X)$ the transcendental lattice of $X$. Denote by $C_n \in \ZZ[x]$ the $n$-th cyclotomic polynomial.

\begin{theorem}[The two prime construction]\label{thm:mcmullen}\cite[7.2]{mcmullen:minimum}
 There exists an automorphism $\lehmer \colon S \rightarrow S$ of a complex projective K3 surface $S$ such that $\NS(S)$ has rank $16$ and discriminant $7 \cdot 13^2$,
 \[\lambda(\lehmer) = \lambda_{10}\mbox{ and }\delta(\lehmer)=\exp(2 \pi i 5/14).\]
 The characteristic polynomial of $\lehmer^*$ on $H^2(S,\ZZ)$ is given by
 \[S_{10}(x)C_{14}(x)C_4(x)(x^2-1)^2.\]
\end{theorem}
\begin{remark}
 From the proof one extracts the following: An even unimodular lattice $H$ and an isometry $l' \in O(H)$ - both as $22 \times 22$ integer matrices. We know that abstractly $H \cong H^2(S,\ZZ)\cong \ZZ^{22}$. A Hodge structure on $H$ is given by setting $H^{2,0}(H_\CC)$ as the eigenspace of $l_\CC$ with eigenvalue $\delta(l)$. By Lefschetz' theorem on $(1,1)$-classes we recover $\NS(S)\cong H^{2,0}(H_\CC)^\perp\cap H$.
\end{remark}
In the following $l\colon S \to S$ will continue to denote the Lehmer map as in \Cref{thm:mcmullen}.
For an even lattice $N$ we denote by $A_{N}$ the discriminant group equipped with the discriminant quadratic form. This is a finite group of order $|\det(N)|$. Recall that we have a natural homomorphism $O(N) \rightarrow O(A_N)$.
Assume that $X$ is projective. Then $\delta(g)$ is a root of unity. Let $\kappa(g)$ denote the kernel of
\[\ZZ[\delta(g)] \rightarrow \End(A_{\NS(X)}),\quad \delta(g) \mapsto (x \mapsto g_*(x)).\]
In particular $\kappa(g)$ is an ideal in a cyclotomic field.

\begin{lemma}\label{lem:unique}
If a complex K3 surface $X$ of Picard number $\rho(X)=16$ admits an automorphism $g$ of order $7$ with $\delta(g)=\delta(l^2)$ and $\kappa(g) = \kappa(l^2)$ then $X$ is isomorphic to McMullen's  surface $S$.
\end{lemma}
\begin{proof}
We take $l^2$ so that $\delta(l^2)$ is a $7$-th root of unity. Then
 \cite[Prop. 5.1]{brandhorst:non-symplectic} applies and yields the result.
\end{proof}
The next lemma shows that such an $X$ in fact exists.
Let $N$ be a lattice and $g \in O(N)$ an isometry.
The fixed lattice and the coinvariant lattice of $g$ are defined as
\[N^g = \{x \in N \mid g(x)=x \} \quad  \mbox{ and } \quad N_g = (N^g)^\perp.\]
By \cite{ast:prime} the fixed lattice of a prime order $p$ automorphism on a K3 surface is $p$-elementary (i.e. $pA_{N^g}=0$). Thus we cannot hope for $\NS(X)=H^2(X,\ZZ)^g$.
 The next best thing is to aim for $H^2(X,\ZZ)^g\cong U \oplus E_8$.
 Then the coinvariant lattice $H^2(X,\ZZ)_g \cong 2U \oplus E_8$.
\begin{lemma}\label{lem:2prime-finite}
There exists a comple K3 surface $X$ which admits an action by a non-symplectic automorphism $g$ of order $7$ with $\NS(X)^g\cong U \oplus E_8$ such that
$\delta(g)=\delta(l^2)$ and $\kappa(g) = \kappa(l^2)$.
\end{lemma}

\begin{proof}
 We obtain $g$ in the style of McMullen \cite{mcmullen:minimum} by equivariant glueing. See
 \Cref{fig:glue} for an overview of the construction. Set $(T,g_T)=(T(S),l_*^2|T(S))$. We obtain
 $(C,g_C)$ as a twist of the principal $\Phi_7(x)$-lattice. By picking a suitable prime of $\ZZ[\zeta_7+\zeta_7^{-1}]$ above $13$ for the twist, one can assure that
 $A_{C} \cong A_T(-1)$ and that
 the characteristic polynomials of the actions on the glue groups
 $(A_{C})_q \cong (A_T(-1))_q$, $q=7,13$ of $C$ and $T$ match.

  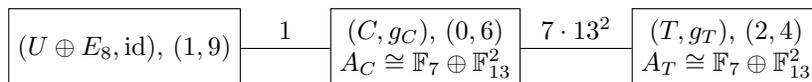
\begin{figure}[!htbp]
  \tikzstyle{block} = [draw, rectangle, minimum height=3em, minimum width=3em]
  \tikzstyle{virtual} = [coordinate]
  \begin{tikzpicture}[auto, node distance=4cm]
    \node [block, align=center]          (S)     {$(U\oplus E_8, \id)$, $(1,9)$};
    \node [block, right of=S, align=center]   (C)     {$(C,g_C)$, $(0,6)$\\ $A_C\cong \FF_7 \oplus \FF_{13}^2$};
    \node [block, right of=C, align=center]   (T)     {$(T,g_T)$, $(2,4)$\\ $A_T\cong \FF_7 \oplus \FF_{13}^2$};
    \draw [-] (C) -- node {$7 \cdot 13^2$} (T);
    \draw [-] (S) -- node {$1$} (C);
  \end{tikzpicture}
  \caption{Gluing diagram for $\sigma$.}\label{fig:glue}
 \end{figure}
  Thus $C$ and $T$ glue equivariantly to an even unimodular lattice of signature $(2,10)$ (cf. \cite[Thm. 4.1]{mcmullen:minimum}). Taking the orthogonal direct sum with $U \oplus E_8$ yields an even unimodular lattice $H$ of signature $(3,19)$. We obtain the isometry $g'=\id \oplus g_C \oplus g_T \in O(H)$. Since $C$ has no roots, $g'$ is unobstructed in the sense of McMullen. By construction it preserves the Hodge structure defined by $H^{2,0}(S)\subseteq T(S)\otimes \CC=T\otimes C$. By surjectivity of the period map and the global Torelli theorem, we find an $H$-marked K3 surface $(X,\eta)$ and an automorphism $g\in \Aut(X)$ with $\eta g_* \eta^{-1}= g'$. By construction $\delta(g)=\delta(l^2)$ and $\kappa(g)=\kappa(l^2)$.
 \end{proof}

Non-symplectic automorphisms of prime order on K3 surfaces are well understood.
Their deformation type is determined by the fixed lattice and for each type we know a maximal family (cf. \cite{ast:prime}). In our case \Cref{lem:2prime-finite,lem:unique} imply that
$S$ must be a member of the one dimensional family $\S$ of complex elliptic K3 surfaces given by
\[\S \colon \quad y^2 = x^3 + ax +t^7+1\]
found in \cite[Example 6.1 \(1\)]{ast:prime}. The automorphism $g$ is given by $(x,y,t) \mapsto (x,y,\zeta_7t)$.

We want to find a member $S_1$ of this family with $\NS(S)\cong \NS(S_1)$.
Since $\triv(S_1)\cong U \oplus E_8$ is unimodular, $\NS(S_1)=\triv(S_1)\oplus \mwl(S_1)(-1)$, that is, the elements of $C$ give elements of the Mordell-Weil group of the fibration.
Note that the minimum of the Mordell-Weil lattice $\mwl(S_1)\cong C(-1)$ is $4$, and in fact there are up to a change of sign exactly $7$ minimal vectors.
\[C =
\begin{bmatrix}
-4 & 1 & 1 & 0 & 0 & 1 \\
1 & -4 & 1 & 1 & 0 & 0 \\
1 & 1 & -4 & 1 & 1 & 0 \\
0 & 1 & 1 & -4 & 1 & 1 \\
0 & 0 & 1 & 1 & -4 & 1 \\
1 & 0 & 0 & 1 & 1 & -4
\end{bmatrix},\quad g_C = \begin{bmatrix}
0 & 0 & 0 & 0 & 0 & -1 \\
1 & 0 & 0 & 0 & 0 & -1 \\
0 & 1 & 0 & 0 & 0 & -1 \\
0 & 0 & 1 & 0 & 0 & -1 \\
0 & 0 & 0 & 1 & 0 & -1 \\
0 & 0 & 0 & 0 & 1 & -1
\end{bmatrix}
\]
They lie in a single $g$-orbit.
 In our case we are searching for a section $P$ of height
$h(P)=4$. This means that $P=(x(t),y(t))$ with $x,y \in K[t]$ of degree at most $4$ resp. $6$.

We reduce the family modulo a prime $p$ and look for the extra sections over $\FF_{p^n}$. We would like to take $n=1$. Since $\S$ admits a non-symplectic automorphism of order $7$ acting non-trivially on the extra sections, it seems likely that $\zeta_7$ is involved in their equations (see \cite{taelman} for a quantitative statement on a field of definition of CM K3 surfaces).
Thus we should assure that $\FF_p$ contains a $7$-th root of unity. This is the case if and only if $7$ divides $p-1$. Hence we continue with $p=29$ and work over $\FF_{29}$.
For $\mathcal{S}/\FF_{29}$ our search is a finite problem.

Since we do not work over an algebraically closed field, it is better to take the general form
\[\S_{a,b,c}\colon \quad y^2 = x^3 + ax + bt^7 + c\]
instead. It can be normalized to the first one but only at the cost of a coordinate change involving roots of the coefficients.
It is equivalent to $\S_{a',b',c'}$ with
\begin{itemize}
 \item $(a',b',c')=(a,v^{7}b,c)$ for $v \in \FF_{29}^\times$ (scale $t$ by $v$),
 \item $(a',b',c')=(m^4a,m^6b,m^6c)$ (scale x by  $m^2$ and $y$ by $m^3$ for $m \in \FF_{29}^\times$).
\end{itemize}
Note that we can assume $b\neq 0$ since otherwise the fibration is trivial.
The case $a=0$ is dealt with separately and we continue with $a\neq 0$.
After a normalization we may assume that $a=b$ and vary $c$ in a set of representatives of $\FF_{29}/(\pm 1)$. (Note that this comes at the price of taking a cube root which is harmless over $\FF_{29}$ but not over $\QQ$.)
We set
\[x(t) = x_0 + x_1 t + x_2 t^2 + x_3t^3 +x_4 t^4.\]
Using the action on $t$ by a $7$-th root of unity we may assume that
$x_4$ varies in a set of representatives of $\FF_{29}/(\FF_{29}^\times)^4$. Further, the leading coefficient must be a square and the degree of $x(t)$ even. But if $\deg x \leq 2$, then $\deg y \leq 3$ and we cannot cancel the $at^7$ term. Thus $\deg x(t)=4$ and we may continue with $x_4=\pm 1$.

Given $(a,c)$ and $x(t)$ we can compute $x^3 + ax + a t^7 +c$ and check whether or not is a square. If it is, we have found a section $P=(x(t),y(t))$ and can compute its square root $y(t)$. This leaves us with at most $2 \cdot 29^4 \cdot 15\cdot 28$ combinations to check.

We can speed this up further by varying $a\in \FF_{29}$ and $x \in \FF_{29}[t]_4$,
then computing $y \in \FF_{29}[t]_6$ such that $x^3+ax+at^7-y^2$ is of degree at most $6$. If this difference is constant, it must be equal to $c$ and we have found a solution. We may assume $x_4=1$ using the coordinate change $(x,y,t)\mapsto (-x,\zeta_4^3 y,-t)$. This way there are only $29^4 \cdot 28$ cases to be checked.

To compute the candidate for $y$, we use the first 6 terms of the power series expansion
\[v=\sqrt{1+u}=1 + 15u + 18u^2 + 20u^3 + 2u^4 + 16u^5 + 17u^6 + O(u^7).\]
Set $g(t) = x^3 + ax + at^7$, which is of degree $12$, and let $1+u = t^{12} g(1/t)$. Then set $y=t^6v(1/t)$ and truncate to get a polynomial of degree $6$.

The computation took about 12 minutes (20 seconds with an optimized implementation on 4 cores) resulting in the
6 fibrations (4 non-isomorphic ones) given by $(a,c) \in \{(2, \pm 11),(14,\pm 1),(19,15),(19,26)\}$.

\begin{remark}\label{rmk:groebner}
 Imposing the existence of an extra section in the family $\S$ with given intersection pairing is a closed condition on the base. This gives an alternative approach. Indeed, write
 write $y(t)^2-x(t)^3+ax(t)+bt^7+c$ which is a polynomial in $t$ of degree 12. Collect the coefficients of $t^i$ for $1 \leq i \leq 12$ to obtain a polynomial system of equations. They generate an ideal in $k[a,b,c,x_0,\dots x_4, y_0, \dots y_6]$ which can be solved using Gr\"obner basis computations. \end{remark}

In a second step we computed, for each candidate $\S_{(a,a,c)}$ and $P$, the Gram matrix of the height pairing with respect to the basis $(P,\sigma^*(P),\dots \sigma^{*5}(P))$. Then we tested if this Gram matrix defines a lattice isometric to $C$.
Since the only reducible fiber is of type $II^*$, i.e. $\tilde{E}_8$, the height pairing is the negative of the intersection pairing.
The height paring of two sections $P \neq Q$ on our surface is given as follows:
write $(\frac{Q_y-P_y}{Q_x-P_x})=\frac{f(t)}{g(t)}$ with $f,g \in K[t]$ coprime polynomials. Then $h(P,Q) = \deg g - 2$.

We obtain determinants $343$ and $448$ for $2$ fibrations each. The remaining two have the desired determinant $1183=7 \cdot 13^2$ and
isomorphic Mordell-Weil lattices.
\begin{eqnarray}
 \S_{19,19,15}\colon \;\; y^2 &=& x^3 + 19x + 19t^7 + 15 \label{newt1}\\
  x(t) &=& t^4 + 7t^3 + 7t^2 + 27t + 16\label{newt2}\\
  y(t) &=& t^6 + 25t^5 + 18t^4 + 25t^3 + 15t^2 + 20t + 23\label{newt3}
\end{eqnarray}
\begin{eqnarray*}
 \S_{19,19,26}\colon \;\; y^2 &=& x^3 + 19x + 19t^7 + 26\\
  x(t) &=& t^4 + 17t^3 + 4t^2 + 22t + 17\\
  y(t) &=& t^6 + 11t^5 + 2t^4 + 18t^3 + 23t^2 + 23t + 19
\end{eqnarray*}

Since $g$ is tame, we know a priori, that the pairs
$(\mathcal{S}_{19,19,15},g)$ and $(\mathcal{S}_{19,19,26},g)$ lift to characteristic zero (cf. \cite{jang}). Each lift supports a map with entropy $\log \lambda_{10}$.  In the next section we carry out the lift for $\mathcal{S}_{19,19,15}$.

\section{A lift to characteristic zero}\label{sect:newton}
Using a $p$-adic multivariate Newton iteration we lift the the surface and its section to characteristic zero as follows.

Regard equation (\ref{newt1}) and its parts (\ref{newt2}-\ref{newt3}) as a solution of
\[ g(z):=g(x,y,a,c)=y^2-x^3-ax-at^7-c=0 \mod p \]
with $x,y\in \ZZ_p[t]$ monic of degree $4$ and $6$.
Let $z_0$ be defined by (\ref{newt1}-\ref{newt3}). We calculate $dg= 2ydy - (3x^2-a)dx-(x+t^7)da-dc$.
Define $z_{n+1}$ recursively as the unique solution of the $12 \times 12$ linear system of equations $0=g(z_n)+dg(z_n)(z_{n+1}-z_n)$.
We continue with the solution $z_{10}$ modulo $p^{2048}$.

The function \textbf{algdep} in \textsc{pari} \cite{pari} discovers an integer polynomial of degree $18$ with root $a$ and a polynomial of degree $6$ with root $c$.
In fact $a^3 \in \QQ[c]$ which reflects that our normalization $a=b$ required taking a cube root. After a change of coordinates (allowing for $a\neq b$) we can get rid of the cube root. Using the function \textbf{polredabs} we find the defining polynomial
\[w^6 - 2w^5 + 2w^4 - 3w^3 + 2w^2 - 2w + 1 = 0\]
of discriminant $7^4 13$ and containing the cubic field $\QQ[\zeta_7 + \zeta_7^{-1}]$ of
discriminant 49. Then $\QQ[w]\cong \QQ[b]$, which is
the number field \href{http://www.lmfdb.org/NumberField/6.2.31213.1}{6.2.31213.1}.
With further coordinate changes we arrive at the following equation for $S=S_1$.

\begin{dgroup}\label{eqn:S}
\begin{dmath*}
S_1\colon  y^2 = x^3 + ax + bt^7 + c \quad \mbox{where}
\end{dmath*}
\begin{dmath*}
a = (-86471w^5 - 19851w^4 - 116626w^3 + 67043w^2 - 125502w + 106947) / 48
\end{dmath*}
\begin{dmath*}
  b = 7  (-w^5 + w^4 + 2w^3 - 3w^2 + 3w - 1)
\end{dmath*}
\begin{dmath*}
  c =(141655682w^5 - 65661512w^4 + 230672148w^3
        - 136877559w^2 + 149096157w - 96818792) / 864
\end{dmath*}
\end{dgroup}
%
The section $P=(x,y)$ is given by
\begin{eqnarray*}
x &=& (w^5 - 2w^4 + w + 1)t^4\\
     &+& (-10w^5 + 12w^4 - 2w^3 + 14w^2 - 2w + 6)t^3\\
     &+& (-22w^5 + 42w^4 - 41w^3 + 53w^2 - 25w + 39)t^2\\
     &+& (-22w^5 + 26w^4 - 41w^3 + 89w^2 - 21w + 68)t  \\
     &+& (-400w^5 + 493w^4 + 98w^3 + 639w^2 + 178w + 557) / 12\\
\end{eqnarray*}
and
\begin{eqnarray*}
  y &=& (-3w^5 + 7w^4 - 6w^3 + 7w^2 - 7w + 4)t^6\\
     &+& (6w^5 - 18w^4 + 9w^3 - 9w^2 + 21w + 9)t^5\\
     &+& (-159/2w^5 + 201/2w^4 - 51w^3 + 153w^2 - 45/2w + 63)t^4\\
     &+& (-206w^5 + 374w^4 - 365w^3 + 537w^2 - 289w + 377)t^3\\
     &+& (-519/2w^5 + 249w^4 - 389/2w^3 + 1203/2w^2 - 14w + 1205/2)t^2\\
     &+& (-1213/2w^5 + 861w^4 - 1153/2w^3 + 1276w^2 - 367/2w + 775)t\\
     &+& (-773/2w^5 + 1553/2w^4 - 745w^3 + 2423/2w^2 - 961/2w + 1347/2)
\end{eqnarray*}

\begin{remark}[The Coxeter construction]
 McMullen constructs a second automorphism $\tilde\lehmer$ of a projective K3 surface $\tilde S$ of dynamical degree
 $\lambda_{10}$. This time $\delta(\tilde\lehmer)$ is a $22$nd root of unity and $\NS(\tilde S)$ has discriminant $11 \cdot 23^2$.
 As before we find that $\tilde S$ lies in a one dimensional family of automorphisms of order $11$,
 we reduced modulo $23$, found $\tilde S/\FF_{23}$ and lifted the surface modulo $23^{8192}$. However, this time
 we were unable to find an algebraic relation up to degree $100$.
 Later, using Gr\"obner bases as in \Cref{rmk:groebner} and the msolve package \cite{msolve}, we found relations of degree $900$. Probably the degree can be lowered by further coordinate changes.
\end{remark}

Now that we have found equations for the surface, we want to find an expression for the Lehmer automorphism.
At this point we have a basis for $\NS(S_1)$ and can compute the corresponding intersection matrix.
On the other hand McMullen's construction
provides a lattice $N$ abstractly isometric to $\NS(S_1)$ and an isometry $l' \in O(N)$ preserving some chamber of the positive cone.
Our next task is to find a concrete isometry $N\cong \NS(S_1)$ and to reconstruct the Lehmer map $l$ from its cohomological shadow $l'$.

\section{Finding a good fibration.}\label{sect:goodfib}
Equations for Lehmer's automorphism $\lehmer$ depend on the coordinates we choose on the surface. In order to get manageable equations, we need to choose suitable coordinates.
We decided to work with Weierstrass models of elliptic fibrations. Let $t$ be a coordinate on the base $\PP^1$ of the fibration.
The intersection $\lehmer_*(F).F$ gives the degree of $\lehmer(t)$ on the generic fiber.
Hence our first step is to search a ``nef divisor'' $f_6 \in N$
with $l'(f_6).f_6$ as small as possible. Here ``nef'' means that $f_6$ is a ray of a chamber preserved by $\lehmer$. Geometrically this means that $f_6$ should be moved as little as possible by $l'$. On the hyperbolic model of the positive cone, $l'$ acts as a translation along a geodesic with start and endpoints the eigenvectors with eigenvalue $1/\lambda_{10}$ and $\lambda_{10}$. These eigenvectors are isotropic but not rational. Thus it is reasonable to search for $f_6$ as a point of small height close to these endpoints. By a computer search we found $f_6$ with $l'(f_6).f_6=4$.

We know that $N$ and $\NS(S_1)$ are abstractly isomorphic since they lie in the same  genus $\even_{(1,15)}7^{-1} 13^{-2}$ which consists of a single isometry class. To find equations for $\lehmer$, we search a sequence of $2$-neighbor steps connecting
$f_6$ with $f_1$. In particular this gives an explicit isometry $\NS(S_1) \rightarrow N$.

We choose some $e_6$ with $e_6^2=-2$ and $e_6.f_6=1$. This results in a hyperbolic plane $U_6$ spanned by $e_6$ and $f_6$ whose orthogonal complement has the root sublattice $\R(U_6^\perp)=A_3 \oplus A_4 \oplus 3 A_1$. The complement $U_6^\perp\subseteq N$ is a negative definite lattice in the genus $\even_{(0,14)}7^{-1} 13^{-2}$. This genus has mass \[15589726949525753/172164671078400= 90.55...\]
which is quite big.
Hence it is better to first exhibit an $E_8$ lattice inside $N$ by hand. This is carried out in a sequence of three $2$-neighbor steps
with root sublattices $\R(U_5^\perp) =A_3 \oplus D_5 \oplus A_1$, $\R(U_4^\perp)=D_8 \oplus A_1$ and finally $\R(U_3)^\perp = E_8 \oplus A_2$.
Giving $U_3\oplus E_8 \oplus C_3 = N$ for some lattice $C_3$ in the genus $\even_{(1,7)}7^{-1} 13^{-2}$.

We continue on our path of $2$-neighbor steps on the orthogonal complement of $E_8$, which is easier since the mass of the genus of $C_3$ is only $2669/2016$ and the dimension of the lattices in question is only $6$.
Indeed the Kneser neighboring algorithm returns $23$ representatives for this genus almost instantly.
One of them is isomorphic to $C$ - the Mordell-Weil lattice of the fibration on $S_1$.
We are searching for a hyperbolic plane $U \subseteq E_8^\perp \subset N$ with $U^\perp$ isomorphic to the frame lattice $E_8 \oplus C$ of the fibration.
With \Cref{lem:neighbor2} we chart a $2$-neighbor path with two steps connecting $U_3^\perp$ and $U_1^\perp$ using reflections where necessary to guarantee nefness.

The path from $U_1$ to $U_6$ yields a sequence of coordinate changes between the Weierstrass models $S_1, \dots S_6$
of the corresponding fibrations. We display the final equation $S_6$. A list of the intermediate fibrations is found in the ancillary files.
Since one needs a primitive $7$th root of unity to get all of $\NS(S_1)$, we worked with the degree $12$ number field $F = \QQ(w,\zeta_7) \cong \QQ(z)$ where
\begin{dgroup}
\begin{dmath}
0=z^{12} - 3z^{11} + z^{10} + 7z^9 - 8z^8 - 4z^7 + 13z^6 - 4z^5 - 8z^4 + 7z^3 + z^2 - 3z + 1
\end{dmath}
\begin{dmath}
w = z^{11} - z^{10} - 3z^9 + 4z^8 + 5z^7 - 6z^6 - 2z^5 + 7z^4 + z^3 - 4z^2 + z + 2
\end{dmath}
\begin{dmath}
\zeta_7 = 2z^{11} - 4z^{10} - 3z^9 + 14z^8 - 4z^7 - 16z^6 + 16z^5 + 8z^4 - 14z^3 + 2z^2 + 6z - 2.
\end{dmath}
\end{dgroup}
Finally a minimal Weierstrass equation for $S_6$ is given by
\[S_6 \colon y^2=x^3 + ax +b\]
where $a$ and $b$ are as follows:
\begin{dgroup*}
\begin{dmath*}
 a = ((-486z^{11}+1620z^{10}-513z^{9}-3888z^{8}+4077z^{7}+2565z^{6}-5616z^{5}+1269z^{4}+3591z^{3}-1620z^{2}-675z+567)t^{8}+(-6696z^{11}+17712z^{10}+432z^{9}-44712z^{8}+30024z^{7}+36720z^{6}-53352z^{5}-2808z^{4}+39960z^{3}-10800z^{2}-10368z+7344)t^{7}+(-2808z^{11}+10044z^{10}-17172z^{9}-1080z^{8}+52488z^{7}-64044z^{6}-29484z^{5}+84024z^{4}-35208z^{3}-37044z^{2}+32076z+5832)t^{6}+(-11016z^{11}+48384z^{10}-63288z^{9}-65664z^{8}+222696z^{7}-81648z^{6}-238248z^{5}+201960z^{4}+56376z^{3}-164592z^{2}+7560z+45360)t^{5}+(-12690z^{11}+93150z^{10}-118260z^{9}-171990z^{8}+411750z^{7}+44280z^{6}-487890z^{5}+222750z^{4}+287820z^{3}-247590z^{2}-90180z+76680)t^{4}+(-1296z^{11}+37800z^{10}-41256z^{9}-86832z^{8}+139320z^{7}+98280z^{6}-151200z^{5}+20088z^{4}+151632z^{3}-23760z^{2}-50112z+7776)t^{3}+(-13284z^{11}+39204z^{10}+19224z^{9}-132732z^{8}+37260z^{7}+183384z^{6}-110052z^{5}-78516z^{4}+152928z^{3}+25272z^{2}-41796z+9180)t^{2}+(-6048z^{11}+2808z^{10}+28944z^{9}-30240z^{8}-46656z^{7}+59184z^{6}+11664z^{5}-63072z^{4}+15336z^{3}+28296z^{2}-7776z-2592)t+297z^{11}-1026z^{10}+3456z^{8}-2889z^{7}-3699z^{6}+5562z^{5}-81z^{4}-4590z^{3}+1782z^{2}+1242z-756),
 \end{dmath*}
 \begin{dmath*}
 b = ((-7722z^{11}+22086z^{10}-3240z^{9}-50652z^{8}+38502z^{7}+40500z^{6}-59562z^{5}-4482z^{4}+45576z^{3}-8856z^{2}-15012z+6804)t^{12}+(-98496z^{11}+212544z^{10}+90072z^{9}-555336z^{8}+109512z^{7}+583200z^{6}-423144z^{5}-344088z^{4}+430272z^{3}+70632z^{2}-174960z+44064)t^{11}+(223560z^{11}-733536z^{10}-90072z^{9}+2659716z^{8}-2026620z^{7}-3154464z^{6}+4465044z^{5}+323352z^{4}-3957336z^{3}+1499148z^{2}+1261980z-763992)t^{10}+(3447576z^{11}-7511184z^{10}-4469040z^{9}+24927480z^{8}-8393328z^{7}-28239408z^{6}+28654992z^{5}+10934352z^{4}-25684344z^{3}+5006664z^{2}+8851896z-3840912)t^{9}+(3403296z^{11}-3695868z^{10}-11826648z^{9}+21801312z^{8}+9769248z^{7}-33295536z^{6}+15361812z^{5}+25153416z^{4}-19040832z^{3}-2238678z^{2}+8672670z-812592)t^{8}+(-4041576z^{11}+18718128z^{10}-16985376z^{9}-32168016z^{8}+60015816z^{7}+7663248z^{6}-64441656z^{5}+30965328z^{4}+37784232z^{3}-26276400z^{2}-8659872z+11021184)t^{7}+(-18053280z^{11}+49242816z^{10}-3374784z^{9}-119968128z^{8}+88202520z^{7}+97750800z^{6}-151035192z^{5}-9475704z^{4}+111685392z^{3}-32377968z^{2}-36442224z+19078416)t^{6}+(-17546544z^{11}+38331792z^{10}+13155048z^{9}-99952056z^{8}+34218288z^{7}+90345456z^{6}-89199792z^{5}-39875328z^{4}+68240880z^{3}-1063368z^{2}-24871536z+5820336)t^{5}+(1819098z^{11}-11564532z^{10}+6563106z^{9}+34937082z^{8}-42080796z^{7}-40589748z^{6}+69748452z^{5}-4833270z^{4}-64619694z^{3}+25145964z^{2}+20500938z-14105664)t^{4}+(8464824z^{11}-19295928z^{10}-8678016z^{9}+58692600z^{8}-20930616z^{7}-64744704z^{6}+63535752z^{5}+26203392z^{4}-58450032z^{3}+7671024z^{2}+20609856z-8338032)t^{3}+(1602828z^{11}-1830924z^{10}-4247964z^{9}+7142256z^{8}+4926420z^{7}-9102456z^{6}+1647540z^{5}+8908056z^{4}-2494800z^{3}-2839860z^{2}+1512432z+366768)t^{2}+(-200232z^{11}+598752z^{10}-12960z^{9}-1662768z^{8}+1236384z^{7}+1578528z^{6}-2214864z^{5}-97848z^{4}+1799496z^{3}-498960z^{2}-505440z+283824)t-6534z^{11}+10422z^{10}+17712z^{9}-45144z^{8}-9180z^{7}+67824z^{6}-36072z^{5}-45360z^{4}+49248z^{3}+3996z^{2}-19440z+6480).
\end{dmath*}
\end{dgroup*}

\section{Equations for Lehmer's map}\label{sect:lehmer}
In this section we factor the Lehmer automorphism of $S$ into a sequence of isomorphisms between elliptic K3 surfaces.
Since birational maps of K3 surfaces are isomorphisms, it suffices to define the map on an open subset.

We will now confirm that the map we constructed has dynamical degree given by Lehmer's number $\lambda_{10}$. As confirming the computations over the degree $12$ field is impractical, we work modulo a prime. This leaves $l_*$ unchanged and gives us the opportunity
to display human readable equations again to illustrate the methods by which we found equations for the Lehmer map.
We reduce modulo the prime  $P=(29,z+12)$ of norm $29$ and work over the residue field $k \cong \FF_{29}$.
The neighbor steps were really carried out over
a number field of degree $12$ (with intermediate steps modulo $P$). The corresponding equations are found in the ancillary files.

\subsection{Preliminaries}
\setcounter{MaxMatrixCols}{20}
After a few neighbor steps we have reached an elliptic fibration given by
\begin{dgroup}
\begin{dmath*}S_6\colon \;y^{2} =x^{3}  +(3t^{8} + 10t^{7} + 6t^{6} + 17t^{5} + 25t^{4} + 4t^{3} + 23t^{2} + 9t + 14)x + 5t^{12} + 25t^{11} + 2t^{10} + 28t^{9} + 28t^{8} + 19t^{7} + 3t^{6} + 17t^{5} + 19t^{4} + 12t^{3} + 25t^{2} + 12t + 6\end{dmath*}
\end{dgroup}

with four sections
\begin{dgroup}
\begin{dmath*}
P_1(t)=(12t^{4} + 21t^{3} + 5t^{2} + 12t + 18 ,\phantom{=} 23t^{5} + 7t^{4} + 22t^{3} + 13t^{2}),
\end{dmath*}
\begin{dmath*}
P_2(t)=(12t^{4} + 20t^{3} + 22t^{2} + 27t + 18 ,\phantom{=} 15t^{4} + 12t^{3} + 12t),
\end{dmath*}
\begin{dmath*}
P_3(t)=(12t^{4} + 20t^{3} + 27t^{2} + 11t + 18 ,\phantom{=} 13t^{4} + 5t^{3} + 26t^{2} + 5t),
\end{dmath*}
\begin{dmath*}
P_4(t)=(4t^{4} + 5t^{3} + 5t^{2} + 6t + 13 , \phantom{=}9t^{6} + 21t^{5} + 17t^{4} + 12t^{2} + 3t + 6 1)
\end{dmath*}
\end{dgroup}
generating its Mordell-Weil group.
Together with the zero section and the reducible fibers they form a basis of $\NS(S_6)$.
Its intersection matrix is the following (the basis elements coming from the fibers can be inferred from the matrix):
\[N_6=\left(
\begin{matrix}
0 &1 &0 &0 &0 &0 &0 &0 &0 &0 &0 &0 &1 &1 &1 &1\\
1&-2 &0 &0 &0 &0 &0 &0 &0 &0 &0 &0 &0 &0 &0 &0\\
0 &0&-2 &1 &0 &0 &0 &0 &0 &0 &0 &0 &0 &1 &1 &0\\
0 &0 &1&-2 &1 &0 &0 &0 &0 &0 &0 &0 &1 &0 &0 &0\\
0 &0 &0 &1&-2 &0 &0 &0 &0 &0 &0 &0 &0 &0 &0 &0\\
0 &0 &0 &0 &0&-2 &1 &0 &0 &0 &0 &0 &0 &0 &0 &0\\
0 &0 &0 &0 &0 &1&-2 &1 &0 &0 &0 &0 &0 &1 &0 &0\\
0 &0 &0 &0 &0 &0 &1&-2 &1 &0 &0 &0 &0 &0 &1 &0\\
0 &0 &0 &0 &0 &0 &0 &1&-2 &0 &0 &0 &1 &0 &0 &0\\
0 &0 &0 &0 &0 &0 &0 &0 &0&-2 &0 &0 &1 &1 &1 &0\\
0 &0 &0 &0 &0 &0 &0 &0 &0 &0&-2 &0 &0 &0 &0 &1\\
0 &0 &0 &0 &0 &0 &0 &0 &0 &0 &0&-2 &0 &0 &1 &0\\
1 &0 &0 &1 &0 &0 &0 &0 &1 &1 &0 &0&-2 &0 &0 &2\\
1 &0 &1 &0 &0 &0 &1 &0 &0 &1 &0 &0 &0&-2 &0 &2\\
1 &0 &1 &0 &0 &0 &0 &1 &0 &1 &0 &1 &0 &0&-2 &2\\
1 &0 &0 &0 &0 &0 &0 &0 &0 &0 &1 &0 &2 &2 &2&-2
\end{matrix}\right)\]
In this basis the pushforward of Lehmer's map $\lehmer$ is given by the following matrix.
\[\lehmer_*=\left(
\begin{matrix}
8 & 0 & 0 & 4 & 4 & 1 & 0 & 1 & 2 & 6 & 0 & 6 & 6 & 1 & 4 & 18 \\
4 & 0 & 0 & 2 & 2 & 0 & 1 & 0 & 1 & 3 & 0 & 3 & 3 & 0 & 2 & 9 \\
-2 & 0 & 0 & -1 & -1 & 0 & 0 & -1 & 0 & -2 & 0 & -2 & -1 & 0 & 0 & -4 \\
-3 & 0 & 0 & -2 & -1 & 0 & 0 & -1 & 0 & -2 & 0 & -2 & -2 & 0 & -1 & -6 \\
-2 & 0 & 0 & -1 & -1 & 0 & 0 & -1 & 0 & -1 & 0 & -1 & -2 & 0 & -1 & -4 \\
-1 & 0 & 0 & 0 & -1 & -1 & 0 & 0 & 0 & -1 & 0 & -1 & -1 & 0 & 0 & -2 \\
-1 & 0 & 0 & 0 & -2 & -1 & 0 & 0 & 0 & -1 & 0 & -1 & 0 & 0 & 0 & -2 \\
-3 & 1 & 0 & -1 & -2 & -1 & 0 & 0 & -1 & -2 & 0 & -2 & -1 & 0 & -1 & -5 \\
-3 & 0 & 0 & -2 & -1 & -1 & 0 & 0 & -1 & -2 & 0 & -2 & -2 & 0 & -2 & -6 \\
-1 & 0 & 0 & 0 & -1 & 0 & 0 & 0 & 0 & -1 & 0 & -1 & -1 & 0 & -1 & -2 \\
0 & 0 & 0 & 0 & 0 & 0 & 0 & 0 & 0 & 0 & 1 & 0 & 0 & 0 & 0 & -1 \\
-2 & 0 & 0 & -1 & -1 & 0 & 0 & 0 & -1 & -1 & 0 & -2 & -2 & -1 & -1 & -4 \\
-1 & 0 & 0 & -1 & 0 & 0 & 0 & 0 & 0 & -1 & 0 & 0 & -1 & 0 & -1 & -2 \\
2 & 0 & 1 & 1 & 0 & 0 & 0 & 0 & 1 & 1 & 0 & 1 & 2 & 0 & 1 & 4 \\
-1 & 0 & 0 & 0 & -1 & 0 & 0 & 0 & -1 & -1 & 0 & -1 & 0 & 0 & 0 & -2 \\
0 & 0 & 0 & 0 & 0 & 0 & 0 & 0 & 0 & 0 & 0 & 0 & 0 & 0 & 0 & -1
\end{matrix}\right)\]
We work with column vectors so that $\lehmer_*^T N_6 \lehmer_* = N_6$.
One may note that $\lehmer_*(f_6).f_6=4$. Instead of performing a (complicated) $4$-neighbor step, we shall perform two $2$-neigbor steps.
An intermediate fibration is given by the following vector:
\[
f_7 = (6, 3, -1, -2, -1, -1, -1, -2, -2, -1, 0, -1, -1, 1, -1, 0).
\]
Indeed $f_7$ is nef and $f_6.f_7=2$, $\lehmer_*(f_6).f_7 = 2$.

\subsection{The first neighbor step}
The fibration on $S_6$ is our starting point. We want to find equations for the fibration $S_7$ corresponding to $f_7$.
We now follow the strategy for fibration hopping given in \Cref{sect:strategy}.

\subsection*{(1-8) The new elliptic parameter} We calculate the linear system $H^0(S,f_7) = \{\phi(c_0,c_1)\mid c_0,c_1 \in \FF_{29}\}$ and
choose $\phi(0,1)/\phi(1,0)=u$ with
\[u=\frac{(-13t^{7} + 5t^{6} + 9t^{5} - 11xt^{3} - 3t^{4} - 11t^{3} - 5xt + yt - 6t^{2} + 2x - 7y + 8t - 7)}{(-10t^{6} + 10t^{5} + 3t^{4} + xt^{2} - 8t^{3} - 14xt - 7t^{2} - 9x + 12t - 12)}\]
as new elliptic parameter.

\subsection*{(9) A curve of genus one}
Solve for $y=y(x,t,u)$ and
set $x\mapsto y'$, $t\mapsto x'$ and $y \mapsto y(y',x',u)$.
We rename $(x,y,t)=(x',y',u)$ and view the resulting change of coordinates $\FF_{29}(x,y,t) \rightarrow \FF_{29}(x,y,t)$ as a bi-rational map of the ambient affine space $\mathbb{A}^3$,
\begin{dgroup}
\begin{dmath*}
x\mapsto y,
\end{dmath*}
\begin{dmath*}
y\mapsto (-13x^{7} + 10x^{6}t + 5x^{6} - 10x^{5}t + 9x^{5} - 3x^{4}t - 3x^{4} - 11x^{3}y + 8x^{3}t - x^{2}yt - 11x^{3} + 7x^{2}t + 14xyt - 6x^{2} - 5xy - 12xt + 9yt + 8x + 2y + 12t - 7)/(-x + 7),
\end{dmath*}
\begin{dmath*}
t\mapsto x,
\end{dmath*}
\end{dgroup}
which leads us to the following equation of a curve of genus $1$ over $\FF_{29}(t)$:
\begin{dmath*}
0=x^{8} + (26t + 24)x^{7} + (19t^{2} + 19t + 3)x^{6} + (10t^{2} + 11t + 2)x^{5} + 24x^{4}y + (3t^{2} + 6t + 1)x^{4} + (22t + 27)x^{3}y + (21t^{2} + 22t + 6)x^{3} + (t^{2} + 11)x^{2}y + (22t^{2} + 12t + 7)x^{2} + (15t^{2} + 10t + 26)xy - y^{2} + (12t^{2} + 13t + 21)x + (20t^{2} + 25t + 17)y + 17t^{2} + 14t + 18.
\end{dmath*}
Next, we complete the square and absorb square factors into $y$.
This transforms it to a hyperelliptic curve (of genus $1$) over $\FF_{29}(t)$:
\begin{dgroup}
\begin{dmath*}
x\mapsto x,
\end{dmath*}
\begin{dmath*}
y\mapsto 12x^{4} + 11x^{3}t - 14x^{2}t^{2} - x^{3} - 7xt^{2} - 9x^{2} + xy + 5xt + 10t^{2} + 13x - 7y - 2t - 6,
\end{dmath*}
\begin{dmath*}
t\mapsto t,
\end{dmath*}
\end{dgroup}
\begin{dmath*}
y^{2}=(7t^{2} + 26t + 20)x^{4} + (11t^{3} + 26t^{2} + 7t + 3)x^{3} + (22t^{4} + 2t^{2} + 15t + 18)x^{2}  + (11t^{4} + 5t^{3} + 10t^{2} + 19t + 5)x + 5t^{4} + 27t^{3} + 11t^{2} + 28t + 23.
\end{dmath*}
\subsection*{(10-11) A Weierstrass model}
To reach a Weierstrass model, we choose the point $(x,y)=(0,-11t^2 + 8t + 9)$ as zero section and move it to infinity:
\begin{dgroup}
\begin{dmath*}
x\mapsto x/y,
\end{dmath*}
\begin{dmath*}
y\mapsto (6x^{3}t^{4} - 14x^{3}t^{3} - 12xyt^{4} - 3y^{2}t^{4} - 10x^{3}t^{2} + 4x^{2}t^{3} + 13xyt^{3} + 7y^{2}t^{3} - 7x^{3}t - x^{2}t^{2} - 3xyt^{2} + 5y^{2}t^{2} - 13x^{3} + 11x^{2}t + 3xyt - 11y^{2}t - x^{2} + 13xy - 8y^{2})/(5y^{2}t^{2} - y^{2}t - 12y^{2}),
\end{dmath*}
\begin{dmath*}
t\mapsto t,
\end{dmath*}
\end{dgroup}
\begin{dmath*}
0=(t^{8} + 5t^{7} + 15t^{6} + 28t^{5} + 20t^{4} + 14t^{3} + 11t^{2} + 26t + 20)x^{3} + (11t^{7} + 3t^{6} + 28t^{5} + 3t^{4} + 26t^{3} + 26t^{2} + 15t + 12)x^{2} + (25t^{8} + 4t^{7} + 2t^{6} + 4t^{5} + 25t^{4} + 9t^{3} + 26t^{2} + 25t + 18)xy + (28t^{8} + 24t^{7} + 14t^{6} + t^{5} + 9t^{4} + 15t^{3} + 18t^{2} + 3t + 9)y^{2} + (27t^{6} + 3t^{5} + 27t^{4} + 28t^{3} + 25t^{2} + 19t + 3)x + (5t^{7} + 24t^{6} + t^{5} + 13t^{4} + 20t^{3} + 10t + 5)y,
\end{dmath*}
\subsection*{(12) Tate's algorithm}
We move two reducible fibers to $0$ and $\infty$,
\begin{dgroup}
\begin{dmath*}
x\mapsto x,
\end{dmath*}
\begin{dmath*}
y\mapsto y,
\end{dmath*}
\begin{dmath*}
t\mapsto (-6t - 11)/t,
\end{dmath*}
\end{dgroup}
\begin{dmath*}
0= (28t^{4} + 4t^{3} + 23t^{2} + 4t + 28)x^{3} + (4t^{6} + 22t^{5} + 3t^{4} + 28t^{2} + t)x^{2} + (12t^{5} + 8t^{4} + 16t^{3} + 16t^{2} + 2t + 4)xy + (t^{4} + 25t^{3} + 6t^{2} + 25t + 1)y^{2} + (25t^{8} + 27t^{7} + 5t^{6} + 12t^{5} + 26t^{4} + 4t^{3} + 12t^{2})x + (5t^{7} + 17t^{6} + 8t^{5} + 28t^{4} + 4t^{3} + 14t^{2} + 11t)y.
\end{dmath*}
Finally, we reach a globally minimal elliptic fibration:
\begin{dgroup}
\begin{dmath*}
x\mapsto (-t^{4} + 9t^{3} - 8t^{2} - 4x - 9t - 11)/(t^{2} - 2t + 1),
\end{dmath*}
\begin{dmath*}
y\mapsto (-11t^{6} - 14t^{5} + 7t^{4} - 5xt^{2} - 13t^{3} + xt + 14t^{2} - 8x + 9y + 12t + 7)/(t^{3} - 3t^{2} + 3t - 1),
\end{dmath*}
\begin{dmath*}
t\mapsto t,
\end{dmath*}
\end{dgroup}
\begin{dmath*}S_7\colon \; y^{2} =x^{3} + (11t^{8} + 15t^{7} + 26t^{6} + 7t^{5} + 14t^{4} + 4t^{3} + 17t^{2} + 15t + 19)x + 20t^{12} + 4t^{11} + 9t^{10} + 15t^{9} + 14t^{8} + 14t^{7} + 16t^{6} + 23t^{4} + 8t^{3} + 5t^{2} + 5t + 1.\end{dmath*}

\subsection*{(13-14) The new basis}
 Next we have to find a basis for $\NS(S_7)$, that is, generators of the Mordell-Weil group of the new fibration.
 Let us call $f_{76}: S_7 \rightarrow S_6$ the map induced by the composition of the rational maps above. We calculate $(f_{76}^{-1})_*(D)$ for $D$ in our basis of $\NS(S_6)$ (since we worked with Weil-Divisors, we cannot simply calculate $f_{76}^*$). The resulting divisors are either fibers or multisections.
 We take the fiberwise trace of each multisection with \Cref{alg:trace} turning it into a section,
 compute the Mordel-Weil lattice in this basis and LLL-reduce it to obtain the following sections of small height:

\begin{dgroup}
\begin{dmath*}
B_1= (25t^{4} + 13t^{3} + 16t^{2} + 23t + 15, \phantom{=}17t^{6} + 3t^{5} + 16t^{4} + 15t^{3} + 5t^{2} + 25t + 6),
\end{dmath*}
\begin{dmath*}
B_2=(25t^{4} + 26t^{3} + 23t^{2} + 4t + 15, \phantom{=}12t^{6} + 6t^{5} + 17t^{3} + 9t^{2} + 4t + 6),
\end{dmath*}
\begin{dmath*}
B_3=(9t^{4} + 4t^{3} + 25t^{2} + 24t + 21,\phantom{=} 23t^{6} + 27t^{5} + 2t^{4} + 7t^{3} + 28t^{2} + 17t + 27),
\end{dmath*}
\begin{dmath*}
B_4=(7t^{4} + 24t^{3} + 27t^{2} + 5t + 19, \phantom{=}18t^{6} + 27t^{5} + 21t^{4} + 26t^{2} + t),
\end{dmath*}
\begin{dmath*}
B_5=(21t^{4} + 20t^{3} + 23t^{2} + 9t + 19, \phantom{=}19t^{5} + 19t^{4} + 17t^{3} + 21t^{2} + 11t).
\end{dmath*}
\end{dgroup}
\begin{remark}
 Calculating $(f_{76}^{-1})_*(D)$ in the degree $12$ field is expensive since $f_{76}$ has large coefficients. So instead we lifted the sections $B_i$ to characteristic zero by a multivariate Newton iteration.
\end{remark}

\subsection*{(15) The intersection matrix}
The intersection matrix $N_7$ of our chosen basis of $\NS(S_7)$ is the following:
 \[N_7=\left(
\begin{matrix}
0 &1 &0 &0 &0 &0 &0 &0 &0 &0 &0 &1 &1 &1 &1 &1\\
1&-2 &0 &0 &0 &0 &0 &0 &0 &0 &0 &0 &0 &0 &0 &0\\
0 &0&-2 &1 &0 &0 &0 &0 &0 &0 &0 &0 &0 &0 &0 &0\\
0 &0 &1&-2 &0 &0 &0 &0 &0 &0 &0 &0 &0 &0 &1 &1\\
0 &0 &0 &0&-2 &1 &0 &0 &0 &0 &0 &0 &0 &0 &0 &1\\
0 &0 &0 &0 &1&-2 &1 &0 &0 &0 &0 &0 &0 &0 &0 &0\\
0 &0 &0 &0 &0 &1&-2 &1 &0 &0 &0 &0 &0 &0 &0 &0\\
0 &0 &0 &0 &0 &0 &1&-2 &0 &0 &0 &0 &0 &0 &0 &0\\
0 &0 &0 &0 &0 &0 &0 &0&-2 &0 &0 &0 &1 &0 &1 &1\\
0 &0 &0 &0 &0 &0 &0 &0 &0&-2 &0 &1 &0 &0 &0 &1\\
0 &0 &0 &0 &0 &0 &0 &0 &0 &0&-2 &0 &0 &1 &0 &0\\
1 &0 &0 &0 &0 &0 &0 &0 &0 &1 &0&-2 &2 &2 &1 &0\\
1 &0 &0 &0 &0 &0 &0 &0 &1 &0 &0 &2&-2 &2 &2 &0\\
1 &0 &0 &0 &0 &0 &0 &0 &0 &0 &1 &2 &2&-2 &2 &2\\
1 &0 &0 &1 &0 &0 &0 &0 &1 &0 &0 &1 &2 &2&-2 &0\\
1 &0 &0 &1 &1 &0 &0 &0 &1 &1 &0 &0 &0 &2 &0&-2
\end{matrix}\right)_.\]
As before one can infer the ordering of fiber components in our basis
from the intersection matrix.

\subsection*{(16) The transformation matrix}
Next, we calculated (the basis representation of) the pushforward
\[(f_{76})_* \colon \quad   \NS(S_7) \longrightarrow \NS(S_6).\]
\[(f_{76})_*=\left(\begin{matrix}
 6&0&0& 1& 4&0&0& 2& 2& 6&0&14&10&10&20& 8\\
 3&0&1& 0& 2&0&0& 1& 1& 3&0& 7& 5& 5&10& 4\\
-1&0&0& 0&-1&0&0& 0& 0&-2&0&-2&-2&-2&-3&-1\\
-2&1&0& 0&-1&0&0&-1& 0&-2&0&-3&-3&-3&-5&-2\\
-1&0&0& 0&-1&0&0&-1& 0&-1&0&-2&-2&-2&-3&-1\\
-1&0&0&-1&-1&0&0& 0& 0&-1&0&-3&-2&-2&-3&-1\\
-1&0&0&-1&-2&1&0& 0& 0&-1&0&-3&-2&-2&-3&-1\\
-2&0&0&-1&-2&0&1&-1&-1&-2&0&-5&-3&-4&-6&-2\\
-2&0&0&-1&-1&0&0&-1&-1&-2&0&-4&-3&-4&-6&-2\\
-1&0&0& 0&-1&0&0& 0& 0&-1&0&-2&-1&-2&-4&-1\\
 0&0&0& 0& 0&0&0& 0& 0& 0&1& 0& 0& 0& 0& 0\\
-1&0&0& 0&-1&0&0& 0&-1&-1&0&-2&-2&-2&-3&-1\\
-1&0&0& 0& 0&0&0&-1& 0&-1&0&-2&-2&-2&-4&-2\\
 1&0&0& 0& 0&0&0& 1& 1& 1&0& 2& 1& 2& 3& 1\\
-1&0&0& 0&-1&0&0& 0&-1&-1&0&-3&-1&-2&-3&-1\\
 0&0&0& 0& 0&0&0& 0& 0& 0&0& 0& 0& 1& 0& 0
\end{matrix}\right)_.\]
Note that $(f_{76})_*^T N_6 (f_{76})_* = N_7$.
\begin{remark}
We represented $f_{76}$ as a
morphism of some open subset of $S_7$ (defined as the complement of the hypersurface of the denominators) to
some affine chart of $S_6$. Many of the basis divisors will not meet the chart used to define $f_{76}$.
However only finitely many do. Thus instead
of only pushing forward the basis, we push forward
elements of the Mordel-Weil group until we have
reached a $\QQ$-basis. The basis representation
of a divisor on $S_6$ was obtained by calculating
the intersection numbers with our given basis.
\end{remark}

\subsection{The second factor}
Using the basis representation we calculate
\[f_8=(f_{76})^{-1}_*\lehmer_*(f_6) = (9, 4, -1, -2, -2, -2, -2, -1, -1, -1, 0, 0, 1, 0, -1, -2).
\]
This is the target elliptic divisor for the second neighbor step.
\subsection*{(1-8) The next elliptic parameter}
The elliptic parameter turns out to be
\begin{dmath}
(5t^{11} + 2t^{10} + 13t^{9} - 3xt^{7} - 5t^{8} - xt^{6} + 7t^{7} + 6xt^{5} - 6yt^{5} - 10t^{6} - 2xt^{4} - 5yt^{4} - 7t^{5} - 5xt^{3} - 11yt^{3} + 7t^{4} - 9xt^{2} + 3yt^{2} - 5t^{3} + 5xt + 4yt - 13t^{2} + 7x + 14t + 10)/(-10t^{11} + 7t^{10} - 13t^{9} + 6xt^{7} - 13t^{8} + 7xt^{6} + 14t^{7} + xt^{5} - 7t^{6} - xt^{4} - 6yt^{4} + 7t^{5} + 10xt^{3} - 5yt^{3} - 6t^{4} - 13xt^{2} - 11yt^{2} + 10t^{3} + 2xt + 3yt - 11t^{2} + 9x + 4y - 10t - 4).
\end{dmath}
\subsection*{(9) A curve of genus one}
As before we use it to derive the coordinate change to the new fibration:
\begin{dgroup}
\begin{dmath*}
x\mapsto y,
\end{dmath*}
\begin{dmath*}
y\mapsto (10x^{11}t + 5x^{11} - 7x^{10}t + 2x^{10} + 13x^{9}t + 13x^{9} + 13x^{8}t - 6x^{7}yt - 5x^{8} - 3x^{7}y - 14x^{7}t - 7x^{6}yt + 7x^{7} - x^{6}y + 7x^{6}t - x^{5}yt - 10x^{6} + 6x^{5}y - 7x^{5}t + x^{4}yt - 7x^{5} - 2x^{4}y + 6x^{4}t - 10x^{3}yt + 7x^{4} - 5x^{3}y - 10x^{3}t + 13x^{2}yt - 5x^{3} - 9x^{2}y + 11x^{2}t - 2xyt - 13x^{2} + 5xy + 10xt - 9yt + 14x + 7y + 4t + 10)/(6x^{5} - 6x^{4}t + 5x^{4} - 5x^{3}t + 11x^{3} - 11x^{2}t - 3x^{2} + 3xt - 4x + 4t),
\end{dmath*}
\begin{dmath*}
t\mapsto x,
\end{dmath*}
\end{dgroup}
\begin{dmath*}0=(3t^{2} + 3t + 27)x^{10} + (8t^{2} + 11t + 26)x^{9} + (10t^{2} + 8t + 20)x^{8} + (26t^{2} + 15t + 26)x^{7} + (4t^{2} + 4t + 4)x^{6}y + (26t^{2} + 16t + 11)x^{6} + (22t^{2} + 3t + 5)x^{5}y + (26t^{2} + t + 22)x^{5} + (4t^{2} + 11t + 28)x^{4}y + (9t^{2} + 4t + 8)x^{4} + (18t^{2} + 24t + 26)x^{3}y + 25x^{2}y^{2} + (23t^{2} + 23t + 20)x^{3} + (14t^{2} + 3t + 22)x^{2}y + 8txy^{2} + (15t + 25)x^{2} + (14t^{2} + 26t + 10)xy + 25t^{2}y^{2} + (19t^{2} + 6t + 3)x + (27t^{2} + 12t + 5)y + 14t + 3\end{dmath*}.
Transform to a hyperelliptic curve over $K(t)$:
\begin{dgroup}
\begin{dmath*}
x\mapsto x,
\end{dmath*}
\begin{dmath*}
y\mapsto (4x^{6}t^{2} + 4x^{6}t - 7x^{5}t^{2} + 4x^{6} + 3x^{5}t + 4x^{4}t^{2} + 5x^{5} + 8x^{4}y + 11x^{4}t - 11x^{3}t^{2} - x^{4} - 3x^{3}y - 5x^{3}t + 14x^{2}t^{2} - 3x^{3} + 5x^{2}y + 3x^{2}t + 14xt^{2} - 7x^{2} - 4xy - 3xt - 2t^{2} + 10x + 14y + 12t + 5)/(8x^{2} + 13xt + 8t^{2}),
\end{dmath*}
\begin{dmath*}
t\mapsto t,
\end{dmath*}
\end{dgroup}
\begin{dmath*}y^{2} =(22t^{4} + 15t^{3} + 16t^{2} + 23t + 7)x^{4} + (12t^{4} + 2t^{3} + 9t^{2} + 22t + 16)x^{3} + (28t^{4} + 28t^{3} + 4t^{2} + 21t + 3)x^{2} + (11t^{4} + 24t^{3} + 3t^{2} + 26t + 18)x + 16t^{4} + 8t^{3} + 21t^{2} + 16t + 13.\end{dmath*}
\subsection*{(10) A rational point}
We change the chart of the hyperelliptic curve and thus find a section
 which was previously at infinity and invisible to us.
\begin{dgroup}\label{eqn:hyperelliptic-chart-at-infinity}
\begin{dmath*}
x\mapsto 1/x,
\end{dmath*}
\begin{dmath*}
y\mapsto y/x^{2},
\end{dmath*}
\begin{dmath*}
t\mapsto t,
\end{dmath*}
\end{dgroup}
\begin{dmath*}y^{2}=(16t^{4} + 8t^{3} + 21t^{2} + 16t + 13)x^{4} + (11t^{4} + 24t^{3} + 3t^{2} + 26t + 18)x^{3} + (28t^{4} + 28t^{3} + 4t^{2} + 21t + 3)x^{2}  + (12t^{4} + 2t^{3} + 9t^{2} + 22t + 16)x + 22t^{4} + 15t^{3} + 16t^{2} + 23t + 7.\end{dmath*}
\subsection*{(11) A Weierstrass model}
Use the section $(x,y)=(0,-14t^2 - 14t - 6)$ to obtain a Weierstrass form. (Note that a different choice of section will result in a different final map!)
\begin{dgroup}
\begin{dmath*}
x\mapsto x/y,
\end{dmath*}
\begin{dmath*}
y\mapsto (-x^{3}t^{3} - 5x^{3}t^{2} + 12xyt^{3} - 14y^{2}t^{3} + 8x^{3}t + 7x^{2}t^{2} + 9xyt^{2} - 12y^{2}t^{2} - 10x^{3} - 14x^{2}t + 7xyt - 4y^{2}t - 9x^{2} + 14xy + 5y^{2})/(y^{2}t + 4y^{2}),
\end{dmath*}
\begin{dmath*}
t\mapsto t,
\end{dmath*}
\end{dgroup}
\begin{dmath*}0=(28t^{6} + 19t^{5} + 20t^{4} + 2t^{3} + 10t^{2} + 15t + 16)x^{3} + (14t^{5} + 13t^{4} + 20t^{3} + 13t^{2} + 9t + 23)x^{2} + (24t^{6} + 22t^{5} + 28t^{4} + 20t^{3} + 5t^{2} + 3t + 19)xy + (t^{6} + 10t^{5} + 9t^{4} + 27t^{3} + 19t^{2} + 14t + 13)y^{2} + (16t^{6} + 20t^{5} + 2t^{4} + 15t^{3} + t^{2} + 21t + 11)x + (11t^{6} + 2t^{5} + t^{4} + 21t^{3} + t^{2} + 5t + 18)y.\end{dmath*}
\subsection*{(12) Tate's algorithm}
Next, we reach a short Weierstrass model (up to scaling):
\begin{dgroup}
\begin{dmath*}
x\mapsto (-4xt^{4} - 8xt^{3} + 10t^{4} + 5xt^{2} - 12t^{3} + 9xt + 12t^{2} + 4x - 3t - 3)/(t^{4} + 2t^{3} + 6t^{2} + 5t - 1),
\end{dmath*}
\begin{dmath*}
y\mapsto (-10xt^{6} + 9yt^{6} - 2xt^{5} - 2yt^{5} + 5t^{6} - 10xt^{4} - 7yt^{4} + 6t^{5} - xt^{3} - yt^{3} + 3t^{4} + 10xt^{2} - 3yt^{2} - 4t^{3} - 6xt + 2yt - 9t^{2} - 14x - 8y + 2t + 10)/(t^{6} + 3t^{5} - 4t^{4} - 13t^{3} - 10t^{2} - 3t + 12),
\end{dmath*}
\begin{dmath*}
t\mapsto t,
\end{dmath*}
\end{dgroup}
\begin{dmath*}0=(t^{12} + 6t^{11} + t^{10} + 8t^{9} + 5t^{8} + 9t^{7} + 23t^{6} + 8t^{5} + 24t^{4} + 9t^{3} + t^{2} + 15t + 28)x^{3} + (28t^{12} + 23t^{11} + 28t^{10} + 21t^{9} + 24t^{8} + 20t^{7} + 6t^{6} + 21t^{5} + 5t^{4} + 20t^{3} + 28t^{2} + 14t + 1)y^{2} + (3t^{12} + 18t^{11} + 6t^{10} + 7t^{9} + 4t^{8} + 10t^{7} + 11t^{6} + 26t^{5} + 23t^{4} + 25t^{3} + 11t^{2} + 10t + 14)x + 25t^{12} + 17t^{11} + 2t^{10} + t^{9} + 20t^{8} + 8t^{7} + 28t^{6} + t^{5} + 13t^{4} + 8t^{3} + 4t^{2} + 7t.\end{dmath*}
Then we produce a globally minimal Weierstrass model:
\begin{dgroup}
\begin{dmath*}
x\mapsto x/(t^{4} + 2t^{3} + 6t^{2} + 5t - 1),
\end{dmath*}
\begin{dmath*}
y\mapsto y/(t^{6} + 3t^{5} - 4t^{4} - 13t^{3} - 10t^{2} - 3t + 12),
\end{dmath*}
\begin{dmath*}
t\mapsto t,
\end{dmath*}
\end{dgroup}
\begin{dmath*} S_8 \colon \; y^{2} =x^{3} + (3t^{8} + 12t^{7} + 22t^{6} + 21t^{5} + 5t^{4} + 8t^{3} + 27t^{2} + 7t + 15)x + 25t^{12} + 17t^{11} + 2t^{10} + t^{9} + 20t^{8} + 8t^{7} + 28t^{6} + t^{5} + 13t^{4} + 8t^{3} + 4t^{2} + 7t.\end{dmath*}
\subsection*{Matching $S_8$ and $S_6$}
By construction the fibered surfaces $S_8$ and $S_6$ are isomorphic. Since we are working with Weierstrass models, one can give such an isomorphism in two steps.
First we apply an automorphism of the base $\PP^1$ to match the singular fibers with those of $S_6$:
\begin{dgroup}
\begin{dmath*}
x\mapsto x,
\end{dmath*}
\begin{dmath*}
y\mapsto y,
\end{dmath*}
\begin{dmath*}
t\mapsto (-8t - 10)/(t - 13),
\end{dmath*}
\end{dgroup}
\begin{dmath*}0=(t^{12} + 18t^{11} + 18t^{10} + 3t^{9} + 21t^{8} + 4t^{7} + 7t^{6} + 9t^{5} + 3t^{4} + 2t^{3} + 27t^{2} + 10t + 23)x^{3} + (28t^{12} + 11t^{11} + 11t^{10} + 26t^{9} + 8t^{8} + 25t^{7} + 22t^{6} + 20t^{5} + 26t^{4} + 27t^{3} + 2t^{2} + 19t + 6)y^{2} + (14t^{12} + 5t^{11} + 4t^{10} + 3t^{9} + 27t^{8} + 9t^{7} + 8t^{6} + 11t^{5} + 5t^{4} + 23t^{3} + 24t^{2} + 18t + 19)x + 6t^{12} + t^{11} + 14t^{10} + 22t^{9} + 22t^{8} + 17t^{7} + 21t^{6} + 3t^{5} + 17t^{4} + 26t^{3} + t^{2} + 26t + 13.\end{dmath*}
Then the resulting elliptic curve is isomorphic to $S_6/\FF_{29}(t)$.
An isomorphism is given by
\begin{dgroup}
\begin{dmath*}
x\mapsto (-13x)/(t^{4} + 6t^{3} - t^{2} - t - 4),
\end{dmath*}
\begin{dmath*}
y\mapsto (-6y)/(t^{6} + 9t^{5} + 12t^{4} - 5t^{3} - 2t^{2} - 7t - 9),
\end{dmath*}
\begin{dmath*}
t\mapsto t,
\end{dmath*}
\end{dgroup}
\begin{dmath*}S_6\colon \; y^{2} =x^{3}  +(3t^{8} + 10t^{7} + 6t^{6} + 17t^{5} + 25t^{4} + 4t^{3} + 23t^{2} + 9t + 14)x + 5t^{12} + 25t^{11} + 2t^{10} + 28t^{9} + 28t^{8} + 19t^{7} + 3t^{6} + 17t^{5} + 19t^{4} + 12t^{3} + 25t^{2} + 12t + 6.\end{dmath*}
We denote by $f_{67}\colon S_6 \rightarrow S_7$ the map induced by the composition of the rational maps above.
\begin{remark}
Our construction assures that $(f_{76})_* (f_{67})_*(f_6) = \lehmer_*(f_6)=f_8$. However, it may happen that $(f_{76})_* (f_{67})_*(o_6)$ is not $ \lehmer_*(o_6)$
but rather another section. This can be compensated by a translation, or by choosing a different zero section when forming the Weierstrass model.
Indeed this is the reason for the change of charts in (\ref{eqn:hyperelliptic-chart-at-infinity}). Once fiber and zero section have the desired pushforward, the number of automorphisms fixing both is finite. They come from automorphisms of the
elliptic curve (fixing the zero section) and automorphisms of the base.
\end{remark}

\subsection* {(13-16) The pushforward}
As before we compute the pushforward $(f_{67})_* \colon \NS(S_6) \rightarrow \NS(S_7)$.
\[(f_{67})_*=\left(
\begin{matrix}
 9&0& 4& 5&0&0&0& 4&0&0&0& 9&10& 2& 4&20\\
 4&0& 2& 2&0&0&0& 2&0&0&0& 4& 5& 1& 2& 9\\
-1&0&-1& 0&0&0&1&-1&0&0&0&-1&-2&-1& 0&-2\\
-2&0&-2& 0&0&1&0&-1&0&0&0&-2&-4&-1& 0&-4\\
-2&0&-3& 0&1&0&0& 0&0&0&0&-2&-5&-1& 0&-4\\
-2&0&-3& 0&0&0&0& 0&0&0&0&-2&-4&-1& 0&-4\\
-2&1&-2& 0&0&0&0& 0&0&0&0&-2&-3&-1& 0&-3\\
-1&0&-1& 0&0&0&0& 0&0&0&0&-2&-2&-1& 0&-2\\
-1&0&-1& 0&0&0&0&-1&1&0&0&-1&-2& 0& 0&-2\\
-1&0&-1& 0&0&0&0& 0&0&1&0&-1&-3&-1&-1&-2\\
 0&0& 0& 0&0&0&0& 0&0&0&1& 0& 0& 0& 0&-1\\
 0&0& 1&-1&0&0&0& 0&0&0&0& 0& 1& 0&-1& 0\\
 1&0& 1& 0&0&0&0& 0&0&0&0& 1& 3& 1& 0& 2\\
 0&0& 0& 0&0&0&0& 0&0&0&0& 0& 0& 0& 0&-1\\
-1&0& 0&-1&0&0&0&-1&0&0&0&-1& 0& 0& 0&-2\\
-2&0&-3& 1&0&0&0& 0&0&0&0&-2&-6&-1& 0&-4
\end{matrix}\right)\]
Note that $(f_{67})_*^T N_7 (f_{67})_* = N_6$.
\subsection*{Lehmer's map}
We confirm that $(f_{76})_* (f_{67})_* = \lehmer_*$, on $\NS(S_6)$.
Thus the composite map $f_{76} \circ f_{67}$ has dynamical degree equal to Lehmer's number.

\bibliographystyle{alpha}
\bibliography{lehmer}

\end{document}